\documentclass[reqno,12pt]{amsart}
\usepackage{amscd}
\usepackage{amssymb}
\usepackage{amsmath}
\usepackage{amsfonts}
\usepackage{enumerate}
\usepackage{graphicx}
\usepackage{epic}
\usepackage[all,cmtip]{xy}
\usepackage{overpic}
\usepackage{subfig}
\usepackage{caption}
\usepackage{tikz}
\usepackage[vcentermath]{youngtab}

\newtheorem{thm}{Theorem}[section]
\newtheorem{prop}[thm]{Proposition}
\newtheorem{lem}[thm]{Lemma}
\newtheorem{cor}[thm]{Corollary}

\theoremstyle{definition}

\newtheorem{rem}[thm]{Remark}
\newtheorem*{rem*}{Remark}

\newtheorem*{nota}{Notations}

\newcommand{\mm}{\mathbf{m}}
\newcommand{\xx}{\mathbf{x}}

\newcommand{\PP}{\mathcal {P}}

\newcommand{\Ss}{\mathcal {S}}

\newcommand{\MM}{\mathcal {M}}

\newcommand{\Zz}{\mathbb {Z}}

\newcommand{\Nn}{\mathbb {N}}

\newcommand{\im}{\mathrm{Im\,}}
\newcommand{\coker}{\mathrm{coker\,}}

\def\w{\widetilde}

\numberwithin{equation}{section}
\usepackage[colorlinks,linktocpage]{hyperref}
\hypersetup{citecolor=blue,linkcolor=blue}
\begin{document}
	\title[Operators on symmetric polynomials...]{Operators on symmetric polynomials and applications in computing the cohomology of $BPU_n$}
	\author[F.~Fan]{Feifei Fan}
	\thanks{The author is supported by the National Natural Science Foundation of China (Grant No. 12271183) and by the GuangDong Basic and Applied Basic Research Foundation (Grant No. 2023A1515012217).}
	\address{Feifei Fan, School of Mathematical Sciences, South China Normal University, Guangzhou, 510631, China.}
	\email{fanfeifei@mail.nankai.edu.cn}
	\subjclass[2020]{55R35, 55R40, 55T10, 05E05}
	\keywords{classifying spaces, projective unitary groups, Serre spectral sequences, Young diagrams, Schur polynomials}
	\maketitle
	\begin{abstract}
	This paper studies the integral cohomology ring of the classifying space $BPU_n$ of the projective unitary group $PU_n$. By calculating a Serre spectral sequence, we determine the ring stucture of $H^*(BPU_n;\Zz)$ in dimensions $\leq 11$.
	For any odd prime $p$, we also determine the $p$-primary subgroups of $H^i(BPU_n;\Zz)$ in the range $i\leq 2p+13$ for $i$ odd and $i\leq 4p+8$ for $i$ even.
    The main technique used in the calculation is applying the theory of Young diagrams and Schur polynomials to certain linear operators on symmetric polynomials. 
	\end{abstract}

	\section{Introduction}\label{sec:introduction}
	
	For a topological group $G$, let $BG$ be its classifying space. This paper studies the integral cohomology of $BPU_n$, where $PU_n$ is the quotient group $U_n/S^1$ of the unitary group $U_n$ by its center $S^1$ embedded as scalars, called the \emph{projective unitary group of degree $n$}.
	
	The cohomology of $BPU_n$ is the main tool to classify principal $PU_n$ bundles and topological Azumaya algebras of degree $n$ over a topological space $X$ (cf. \cite{Gro68}). It also plays a key role in the topological period-index problem introduced by Antieau-Williams \cite{AW14a,AW14b}, and further studied by Gu \cite{Gu19,Gu20} and Crowley-Grant \cite{CG20}. However, its calculation is known as a difficult problem in algebraic topology. For special values of $n$, $H^*(BPU_n;R)$ with $R=\Zz$ or a finite field  has been studied by reaserchers in various ways. Here we give a review of these works.

	A general way to calculate the ordinary cohomology $H^*(BG;R)$ for a topological group $G$ and a commutative ring $R$ is to use the Eilenberg-Moore spectral sequence of the universal principle $G$-bundle $G\to EG\to BG$
	\[E_2=\mathrm{Cotor}_{H^*(G;R)}(R,R)\Longrightarrow H^*(BG;R),\]
	by the knowledge of the Hopf algebra structure of $H^*(G;R)$. For $p$ a prime, the Hopf algebra structure of $H^*(PU_n;\Zz/p)$ is determined by Baum-Browder \cite{BB65}, and the module structure of $H^*(PU_n;\Zz/2)$ when $n\equiv 2$ mod $4$ is determined as the collapsing of the above spectral
	sequence \cite{KM75}. Also, the cohomology ring $H^*(BPU_3;\Zz/3)$ is computed  this way \cite{KMS75}.
	This strategy seems no longer works for the calculation of $H^*(BPU_n;\Zz)$  because the Hopf algebra structure of $H^*(PU_n;\Zz)$ is unkonwn for general $n$, even though the ring structure of $H^*(PU_n;\Zz)$ is determined by Duan \cite{Duan20}. 
	
	Another way to calculate the cohomology of $BPU_n$ is to use the fibration $BS^1\to BU_n\to BPU_n$. By computing the Eilenberg-Moore spectral sequence of this fibration, Toda \cite{Tod87} determined the ring structure of $H^*(BPU_{4m+2};\Zz/2)$. In \cite{Tod87}, Toda also computed the cohomology ring $H^*(BPU_4;\Zz/2)$.
    Using Toda's result for $BPU_4$, together with some partial computations of the Serre spectral sequence introduced below, the author \cite{Fan24} determined the ring structure of $H^*(BPU_4;\Zz)$.

	Another important result  is given by Vistoli \cite{Vis07}, where he provided a nice description of the integral cohomology ring of $BPU_p$ for any odd prime $p$. For $p=2$, we have $PU_2=SO_3$ and the ring $H^*(BSO_3;\Zz)$ is well known.
	Vistoli's proof uses many techniques in algebraic geometry, some of which were originally developed by Vezzosi \cite{Vezz00}. But his proof does not apply to general $n$. Some other  interesting results on the mod $p$ cohomology ring of $BPU_p$, as
	well as its Brown-Peterson cohomology can be found in \cite{VV05} and  \cite{KY08}. 
	
	On the other hand, the computation of $H^*(BPU_n;\Zz)$ for an arbitrary $n$ seems very difficult, and very little is known. For $k\leq 5$,  $H^k(BPU_n;\Zz)$ was computed by Antieau-Williams \cite{AW14b}.
	
	Recently, Gu \cite{Gu21} developed a new method to compute the integral cohomology of $PU_n$ by use of the Serre spectral sequence associated to the fibration
	\[BU_n\to BPU_n\to K(\Zz,3),\]
	where $K(\Zz,3)$ denotes the Eilenberg-Mac Lane space with the third
	homotopy group $\Zz$. Although this spectral sequence does not collapse, it has the advantage that the higher differentials in certain range of the total degrees can be fully described. Then calculations of this spectral sequence detect the cohomology of $PU_n$ in a finite range of dimensions. Using this method, Gu determined the ring structure of $H^*(BPU_n;\Zz)$ in dimensions  $\leq 10$ for any value of $n$. So far, this is the best known result on this problem. In \cite{CG20}, Crowley-Gu studies the image of the homomorphism $H^*(BPU_n;\Zz)\to H^*(BU_n;\Zz)$, which gives a nice description to the quotient ring of $H^*(BPU_n;\Zz)$ by torsion elements.
	Other subsequent results on the torsion subgroup of $H^*(BPU_n;\Zz)$ can be found in \cite{GZZZ22,ZZZ23,ZZ24}.
	
	In this paper we follow Gu's strategy to do more computations of this Serre spectral sequence, and provide more information on $H^*(BPU_n;\Zz)$. 
	\begin{nota}
		Throughout the rest of this paper, we use the simplified notation $H^*(-)$ to denote the integral cohomology $H^*(-;\Zz)$.  For an abelian group $G$ and a prime number $p$, let  $G_{(p)}$ denote the localization of $G$ at $p$, and let $_pG$ denote the $p$-primary component of $G$.
	\end{nota}
	
	Note that $PU_n$ is also the quotient group of the special unitary group $SU_n$ by its center $\Zz/n$. So there is an induced fibration of classifying spaces:
	\[B(\Zz/n)\to BSU_n\to BPU_n.\]
	For a prime $p$ not dividing $n$, the space $B(\Zz/n)$ is $p$-locally contractible, therefore from the Serre spectral sequence associated to this fibration we see that
	\[H^*(BPU_n;\Zz_{(p)})\cong H^*(BSU_n;\Zz_{(p)})\cong \Zz_{(p)}[c_2,c_3\dots,c_n],\ \ \deg(c_i)=2i.\] 
	So the rank of $H^*(BPU_n)$ is the same as the rank of the graded polynomial algebra $H^*(BSU_n)$, and to determine the graded group structure of $H^*(BPU_n)$, it
	suffices to consider the $p$-primary subgroups $_pH^*(BPU_n)$ for all prime divisors $p$ of $n$. Our first main result is the following
	
\begin{thm}\label{thm:p-torsion}
Let $p$ be an odd prime, and $n=p^rm$, $r>0$, for a positive integer $m$
co-prime to $p$. Then the $p$-primary subgroup of $H^*(BPU_n)$ satisfies

\begin{enumerate}
	\item\label{item:1} In odd dimensions $\leq 2p+13$, we have 
	\[\begin{split}
&_pH^{2i+1}(BPU_n)=\begin{cases}
	\Zz/p^r,&i=1,\\
	\Zz/p,&i=p+2,\\
	0,&i\leq p+5,\ i\neq 1,\,p+2,\,p+4,
\end{cases}\\
&_pH^{2p+9}(BPU_n)=\begin{cases}
	\Zz/p,&\text{if }p=3\text{ and }r=1,\\
	0,&\text{otherwise},
\end{cases}\\
&_pH^{2p+13}(BPU_n)=\begin{cases}
	\Zz/p,&\text{if }p=3,\\
	0,&\text{otherwise}.
\end{cases}
	\end{split}\]
	
	\item\label{item:2} In even dimensions $\leq 4p+8$, we have
	\[\begin{split}
	&_pH^{2i}(BPU_n)=\begin{cases}
		\Zz/p,&i=p+1,\,2p+2\\
		0,&i\leq 2p+3,\ i\neq p+1,\,2p+2.
	\end{cases}\\
	&_pH^{4p+8}(BPU_n)=\begin{cases}
		\Zz/p,&\text{if }p=3,\\
		0,&\text{if }p>3.	
	\end{cases}
	\end{split}\]	
\end{enumerate}
\end{thm}
\begin{rem}\label{rem:context}
	 For $k\leq 2p+8$, the groups $_pH^k(BPU_n)$ are already known. They are given in \cite{Gu21} for $k\leq 2p+1$, in \cite{GZZZ22} for $2p+2\leq k\leq 2p+4$, in \cite{ZZZ23} for $k=2p+6$, and in \cite{ZZ24} for $k=2p+5$, $2p+7$, $2p+8$.
\end{rem}

The ring structure of $H^*(BPU_n)$ is more interesting than its group structure. 
Our second theorem improves Gu's result by raising the bound of dimensions to $11$.
\begin{thm}\label{thm:below 12}
For an integer $n>1$, $H^{*}(BPU_n)$ in dimensions $\leq 11$ is isomorphic
to the following graded ring:
\[\Zz[\alpha_2,\dots,\alpha_{j_n},x_1,y_{3,0},y_{2,1}]/I_n.\]
Here $j_n=\min\{5,n\}$, $\deg(\alpha_i)=2i$, $\deg(x_1)=3$, $\deg(y_{p,i})=2p^{i+1}+2$. The ideal $I_n$ of relations is generated by
\begin{gather*}
	nx_1,\ \epsilon_2(n)x_1^2,\ \epsilon_3(n)y_{3,0},\ \epsilon_2(n)y_{2,1},\ \delta(n)\alpha_2x_1,\\
	(\delta(n)-1)(y_{2,1}-\alpha_2x_1^2),\ \alpha_3x_1,\ \mu(n)\alpha_4x_1,
\end{gather*}
where $\epsilon_p(n)=\gcd(p,n)$ and 
\[\delta(n)=\begin{cases}
	2,&\text{if }n\equiv 2 \textrm{ mod } 4,\\
	1,&\text{otherwise,}
\end{cases}
\quad \mu(n)=\begin{cases}
	1,&\text{if }n\not\equiv 0 \textrm{ mod } 4,\\
	4,&\text{if }n\equiv 4 \textrm{ mod } 8,\\
	2,&\text{if }n\equiv 0 \textrm{ mod } 8.
\end{cases}\]
\end{thm} 
Here is the organization of this paper. In Section \ref{sec:SS} we recall the Serre spectral sequence $^UE$ associated to the fibration $BU_n\to BPU_n\to K(\Zz,3)$, which is the main object studied in \cite{Gu21}. In Section \ref{sec:Schur} we recall some classical results in the theory of Young diagrams and Schur polynomials that we will use to calculate certain  differentials in the spectral sequence $^UE$. As we will see, the first and second nontrivial differentials in the localization of $^UE$ at an odd prime $p$ can be identified with two linear operators on symmetric polynomials, for which Schur polynomials provid a good basis. The main calculations about these two operators are contained in Section \ref{sec:cokernel} to \ref{sec:Gamma}, and the proof of Theorem \ref{thm:p-torsion} and \ref{thm:below 12} are given in Section \ref{sec:proof} and \ref{sec:proof below 12}, respectively.

\section{A spectral sequence converging to $H^*(BPU_n)$}\label{sec:SS}

The short exact sequence of Lie groups
	\[1\to S^1\to U_n\to PU_n\to 1\]
	induces  a fibration sequence of their classifying spaces 
	\[BS^1\to BU_n\to BPU_n.\]
	Notice that $BS^1$ is of the homotopy type of the Eilenberg-Mac Lane space $K(\Zz,2)$. So we obtain another fibration sequence:
	\[U:BU_n\to BPU_n\to K(\Zz,3).\]
	(Cf. \cite{Gan67}.) Let $^UE$ be the integral cohomological Serre spectral sequence associated to this fibration. The $E_2$ page of this spectral sequence has the form
	\[^UE_2^{s,t}=H^s(K(\Zz,3);H^t(BU_n))\Longrightarrow H^{s+t}(BPU_n).\]
	We will use this spectral sequence to compute the cohomology of $BPU_n$. In order to carry out the computation, we first need to know the cohomology of $K(\Zz,3)$.
	
	The integral cohomology of the Eilenberg-Mac Lane space $K(\pi,n)$ for $\pi$ a finitely generated abelian group can be computed by using Cartan's method developed in \cite[Expos\'e 11]{Car54}.	For the special case that $\pi=\Zz$ and $n=3$, Gu \cite{Gu21} gave a more explicit description of $H^*(K(\Zz,3))$.
	
	Throughout the rest of this section, we write $\Zz[x;k]$, $E[y;k]$ for the polynomial algebra and exterior algebra over $\Zz$ respectively, with one generator of degree $k$. 
	
	Let $C(3)$ be the (differential graded) $\Zz$-algebra
	\[\begin{split}
		C(3)&=\bigotimes_{p\neq 2\text{ prime}}\big(\bigotimes_{k\geq0}\Zz[y_{p,k};2p^{k+1}+2]\otimes E[x_{p,k};2p^{k+1}+1]\big)\\
		&\otimes\bigotimes_{k\geq0}\Zz[x_{2,k};2^{k+2}+1]\otimes\Zz[x_1;3].
	\end{split}\]
	For notational convenience, we set $y_{2,0}=x_1^2$ and $y_{2,k}=x_{2,k-1}^2$ for $k\geq 1$. The differential of $C(3)$ is given by
	\[d(x_{p,k})=py_{p,k},\ d(y_{p,k})=d(x_1)=0,\text{ for } p \text{ a prime and }k\geq0.\]
	Gu \cite{Gu21} showed  that the integral cohomology ring of $K(\Zz,3)$ is isomorphic to the cohomology of $C(3)$. 
	
	The following two propositions are partial results of \cite[Proposition 2.14]{Gu21}.
	
	\begin{prop}[{\cite[Corollary 2.15]{Gu21}}]\label{prop:K_3 below 14}
		In dimensions $\leq 14$, $H^*(K(\Zz,3))$ is isomorphic to the following graded ring:
		\[\Zz[x_1,y_{2,1},y_{3,0},y_{5,0}]/(2x_1^2,2y_{2,1},3y_{3,0},5y_{5,0}).\]
	\end{prop}
	
	\begin{prop}\label{prop:K_3 odd prime}
		Let $p$ be an odd prime. Then in dimensions $\leq 2p^2+2p+2$, the $p$-local cohomology $H^*(K(\Zz,3);\Zz_{(p)})$ is isomorphic to the following graded ring:
		\[\Zz_{(p)}[x_1,y_{p,0},y_{p,1}]/(x_1^2,py_{p,0},py_{p,1}).\]
	\end{prop}

	To compute some higher differentials of the spectral sequence $^UE$, we consider two other spectral sequences.
Let $T^n$ and $PT^n$ be the maximal tori of $U_n$ and $PU_n$ respectively, and let $\psi:T^n\to U_n$ and $\psi':PT^n\to PU_n$ be the inclusions.
	Then there is an exact sequence of Lie groups 
\[1\to S^1\to T^n\to PT^n\to 1,\]
which as above induces a fibration sequence 
\[T:BT^n\to BPT^n\to K(\Zz,3).\]
We also consider the path fibration sequence
\[K: BS^1\simeq K(\Zz,2)\to *\to K(\Zz,3),\]
where $*$ denotes a contractible space. We denote the Serre spectral
sequences associated  to $T$ and $K$ as $^TE$ and $^KE$ respectively.

Let $\varphi:S^1\to T^n$ be the diagonal map. Then there is the following homotopy commutative diagram of fibration sequences
\begin{gather}\label{diag:SS}
\begin{gathered}
	\xymatrix{K:\ar[d]^{\Phi}& BS^1\ar[r] \ar[d]^{B\varphi}&\ast\ar[r] \ar[d]&K(\Zz,3)\ar[d]^{=}\\
		T:\ar[d]^{\Psi} & BT^n\ar[r] \ar[d]^{B\psi}&BPT^n\ar[r] \ar[d]^{B\psi'}&K(\Zz,3)\ar[d]^{=}\\
		U: & BU_n\ar[r]&BPU_n\ar[r]&K(\Zz,3)}
\end{gathered}
\end{gather}

We write $^Ud_*^{*,*}$, $^Td_*^{*,*}$, $^Kd_*^{*,*}$ for the differentials of $^UE$, $^TE$, and $^KE$, respectively. We also use the simplified notation $d^{*,*}_*$ if there is no ambiguity.

Recall that the cohomology rings of the fibers $BS^1$, $BT^n$ and $BU_n$ are
\begin{align}
	H^*(BS^1)&=\Zz[v],\ \deg(v)=2,\\
	H^*(BT^n)&=\Zz[v_1,\dots,v_n],\ \deg(v_i)=2,\\
	H^*(BU_n)&=\Zz[c_1,\dots,c_n],\ \deg(c_i)=2i.\label{eq:BPUn}
\end{align}
Here $c_i$ is the $i$th universal Chern class of the classifying space $BU_n$ for $n$-dimensional complex bundles.
The induced homomorphisms between these cohomology rings are given by
\[
	B\varphi^*(v_i)=v,\quad B\psi^*(c_i)=\sigma_i(v_1,\dots,v_n),
\]
where $\sigma_i(v_1,\dots,v_n)$ is the $i$th elementary symmetric polynomial in the variables $v_1,\dots,v_n$. Hence, from \eqref{diag:SS} we get the maps $\Phi^*: {^T}E_2^{*,*}\to {^K}E_2^{*,*}$ and $\Psi^*:{^U}E_2^{*,*}\to {^T}E_2^{*,*}$ of spectral sequences on $E_2$ pages. Since these spectral sequences are concentrated in even rows,  all $d_2$ differentials are trivial, so their $E_3$ pages are equal to the $E_2$ pages.

The differentials of the spectral sequence $^KE$ are  described in \cite{Gu21} as follows. Here we only list a part of the results in {\cite[Corollary 2.16]{Gu21}}, which will be sufficient for our purposes.

\begin{prop}[{\cite[Corollary 2.16]{Gu21}}]\label{prop:differentials K}
	For a prime $p$, the higher differentials of $^KE^{*,*}_*$ satisfy
	\[\begin{split}
		&d_3(v)=x_1,\\
		&d_{2p^{k+1}-1}(p^kx_1v^{lp^e-1})=v^{lp^e-1-(p^{k+1}-1)}y_{p,k},\ e>k\geq 0,\ \gcd(l,p)=1,\\
		&d_r(x_1)=d_r(y_{p,k})=0\text{ for all }r\geq2,\,k\geq0,
	\end{split}\]
	and the Leibniz rule.
	Here $^KE^{3,2(lp^e-1)}_{2p^{k+1}-1}\subset {^KE}^{3,2(lp^e-1)}_2$ is generated by $p^kx_1v^{lp^e-1}$.
\end{prop}

To describe the differential of $^TE^{*,*}_*$, we first rewrite  $H^*(BT^n)=\Zz[v_1,\dots,v_n]$ as $\Zz[v_1-v_n,\dots,v_{n-1}-v_n,v_n]$. Then an element in $^TE^{*,*}_*$ can be written as $f(v_n)\xi$, where $\xi\in H^*(K(\Zz,3))$ and
$f(v_n)=\sum_ia_iv_n^i$ is a polynomial in $v_n$ with coefficients $a_i\in\Zz[v_1-v_n,\dots,v_{n-1}-v_n]$. Let $v_j'=v_j-v_n$ for $1\leq j\leq n-1$, and write 
$a_i=\sum_{t_1,\dots,t_{n-1}}k_{i,t_1,\dots,t_{n-1}}(v'_1)^{t_1}\cdots (v_{n-1}')^{t_{n-1}}$, $k_{i,t_1,\dots,t_{n-1}}\in\Zz$.
\begin{prop}[{\cite[Proposition 3.3 and Remark 3.7]{Gu21}}]\label{prop:T}
In the above notation, the differential $^Td_r^{*,*}$ is given by 
\[^Td_r^{*,*}(f(v_n)\xi)=\sum_i\sum_{t_1,\dots,t_{n-1}} {^K}d_r^{*,*}(k_{i,t_1,\dots,t_{n-1}}\xi v_n^i)\cdot(v'_1)^{t_1}\cdots (v_{n-1}')^{t_{n-1}},\]
where where ${^K}d_r(k_{i,t_1,\dots,t_{n-1}}\xi v_n^i)$ is simply ${^K}d_r(k_{i,t_1,\dots,t_{n-1}}\xi v^i)$ with $v$ replaced by $v_n$. 
In particular, the differential $^Td_3^{0,t}$ is given by the ``formal devergence''
\[\nabla:=\sum_{i=1}^n(\partial/\partial v_i):H^t(BT^n)\to H^{t-2}(BT^n)\]
in such a way that $^Td^{0,*}_3=\nabla(-)\cdot x_1$.
Moreover, the spectral sequence $^TE^{*,*}_*$ degenerates at \[
^TE^{0,*}_4={^TE}^{0,*}_\infty=\Zz[v_1',\dots,v_{n-1}'].
\] 
\end{prop}

Here is a corollary of Proposition \ref{prop:T}.

\begin{cor}[{\cite[Corollary 3.4]{Gu21}}]\label{cor:U^d_3}
	For $c_k\in H^*(BU_n)\cong {^U}E_3^{0,*}$, $1\leq k\leq n$, we have
	\begin{equation}\label{eq:nabla}
		^Ud_3(c_k)=(n-k+1)c_{k-1}x_1.
	\end{equation}
\end{cor}
The following result of Crowley-Gu says that $^Ud_r^{0,*}$ are trivial for $r\geq 4$.
\begin{thm}[{\cite[Theorem 1.3]{CG21}}]\label{thm:E_4}
	$^UE_4^{0,*}={^U}E_\infty^{0,*}\cong H^*(BPU_n)/torsion.$
\end{thm}

Let $\chi:BPU_n\to K(\Zz,3)$ be the fibration. For simplicity, we abuse the notations $x_1$, $y_{p,k}$ in $H^*(K(\Zz,3))$ to denote $\chi^*(x_1)$, $\chi^*(y_{p,k})$ in $H^*(BPU_n)$ whenever there is no risk of ambiguity.
The following result of Gu will be useful for the computation of $^UE$.

\begin{thm}[{\cite[Theorem 1.1]{Gu21b}}]\label{thm:torsion elements}
	$y_{p,k}\neq 0$ in $H^*(BPU_n)$ for all odd prime divisors $p$ of $n$ and $k\geq 0$.  
\end{thm}
	
	\section{Some results on Schur polynomials}\label{sec:Schur}
	\subsection{Young diagrams/partitions and Schur polynomials} Young diagrams are indexed by integer partitions.
	Given a partition $\lambda=(\lambda_1,\lambda_2,\dots,\lambda_k)$ of a positive integer, where $\lambda_1\geq\lambda_2\geq\dots\geq\lambda_k$ and each $\lambda_i$ is a positive integer, the \emph{Young diagram} corresponding to $\lambda$, which we still denote by $\lambda$, is the collection of boxes, arranged in left-justified rows, with the length of the $i$th row is $\lambda_i$. For example, the partition $(4, 2, 1)$ of $7$ corresponds to the Young diagram
	\[\yng(4,2,1)\]
	 Let $\mathcal{Y}$ be the set of Young diagrams.  It is often convenient to add one or more zeros at the end of the partition sequence, and to identify sequences that differ only by such zeros. The notation $\lambda\vdash n$ is used to say that $\lambda$ is a partition of $n$, and $\#\lambda$ denote the number of parts (rows) of the partition (Young diagram) $\lambda$. For each partition $\lambda$, the \emph{conjugate} of $\lambda$, denoted $\bar\lambda$, is the partition by listing the number of boxes of the Young diagram $\lambda$ in each column. 
	 So, the conjugate of the above partition is $(3,2,1,1)$. We write $\mu\subset\lambda$ if the Young diagram of $\mu$ is contained in that of $\lambda$.
	 
	 $\PP(n)$, the \emph{partition function}, denote the number of different partitions of a positive integer $n$. Set $\PP(0)=1$ by convention.

	A \emph{(Young) tableau} is obtained by filling in the boxes of the Young diagram with positive integers, e.g. $\young(2134,52,3)$. A tableau is called \emph{semistandard}, if the entries weakly increase along each row and strictly increase down each column. For examplex, $\young(1224,23,3)$ is a semistandard tableau.
	
	For a partition $\lambda=(\lambda_1,\lambda_2,\dots,\lambda_n)$, the \emph{Schur polynomial} $s_\lambda(x_1,x_2,\dots,x_n)$ in $n$ variables is a sum of monomials:
	\[s_\lambda(x_1,x_2,\dots,x_n)=\sum_T x^T=\sum_T x_1^{t_1}\cdots x_n^{t_n},\]
	where the summation is over all semistandard tableaux $T$ of shape $\lambda$ using the numbers from $1$ to $n$, and the exponent $t_i$ is the number of times $i$ occurs in $T$. 
	
	Here we recall some well-known properties of Schur polynomials. For an integer $n>0$, let $\Lambda_n$ be the ring of symmetric polynomials over $\Zz$ in $n$ variables. It is known that $\Lambda_n=\Zz[\sigma_1,\dots,\sigma_n]$, where $\sigma_i$ is the $i$th elementary symmetric polynomial in $n$ variables.  For $\Lambda_n$, let $\Lambda_n^d$ be the $\Zz$-submodule of symmetric polynomials of degree $d$.
	\begin{thm}[see \cite{Ful97}]\label{thm:schur}
Schur polynomials are symmetric polynomials. Let $\MM_n^k$ be the set of partitions of $k$ with at most $n$ parts. Then the Schur polynomials $\{s_\lambda:\lambda\in\MM_n^k\}$ form a basis for $\Lambda_n^k$.
	\end{thm}

	For a Young diagram $\lambda$ and a box $(i,j)\in\lambda$, in the 
	$i$th row and $j$th column, the \emph{hook length} at $(i,j)$, denoted $h_\lambda(i,j)$, is the number of boxes directly below or to
	the right of $(i,j)$ (counting $(i,j)$ once). For example, $h_{(4,2,1)}(1,2)=4$.
	\begin{thm}[see {\cite[Theorem 7.21.2]{Sta99}}]\label{thm:hook}
	For a partition $\lambda=(\lambda_1,\lambda_2,\dots,\lambda_n)$, the number of monomial terms in the Schur polynomial $s_\lambda(x_1,x_2,\dots,x_n)$  is given by
	\[s_\lambda(1,1,\dots,1)=\prod_{(i,j)\in\lambda}\frac{n-i+j}{h_\lambda(i,j)}.\]
    \end{thm}
    \begin{rem}
    	The formula in Theorem \ref{thm:hook} is the so-called \emph{hook content formula}, which was first obtained by Stanley \cite{Sta71}.
    \end{rem}
	
    \subsection{The operator $\nabla$ on Schur polynomials} Let $\nabla=\sum_{i=1}^n\partial/\partial x_i$ be the linear differential operator acting on the polynomial ring $P_n=\Zz[x_1,\dots,x_n]$.
    Clearly, $\nabla$ preserves symmetric polynomials, so $\nabla$ can be restricted to $\Lambda_n$. 
    Generally, for any commutative ring $R$ with unit, we use the notation $\nabla_R$ to denote the map $\nabla\otimes R:P_n\otimes R\to P_n\otimes R$ or $\Lambda_n\otimes R\to \Lambda_n\otimes R$. In particular, when $R=\Zz_{(p)}$ for a prime $p$, we use the abbreviate notation $\nabla_{(p)}$ for $\nabla_{\Zz_{(p)}}$. $\nabla^i_R$ means the restriction of $\nabla$ to $P_n^i$ or $\Lambda_n^i$.
    The following result in \cite{Nen20} gives a nice description of $\nabla$ on Schur polynomials.
    \begin{thm}[{\cite[Theorem 7]{Nen20}}]\label{thm:nabla}
    	For any Young diagram $\lambda$ with at most $n$ rows, we have:
    	\[\nabla(s_\lambda(x_1,\dots,x_n))=\sum_{\substack{(i,j)\in\mathbb{N}^2\\\lambda'=\lambda-(i,j)\in\mathcal{Y}}}(n+j-i)s_{\lambda'}(x_1,\dots,x_n).\]
    \end{thm}
    \begin{rem}
    The formula in Theorem \ref{thm:nabla} can be seen as a special case of the formula for $\nabla$ on \emph{Schubert polynomials} in \cite{HPSW20}. Note that Schur polynomials are special cases of Schubert polynomials. We refer the reader to \cite{Ful97} for background on Schubert polynomials.
    \end{rem}

\section{The cokernel of $\nabla$ on $\Lambda_n$ in low degrees} \label{sec:cokernel}  
	The following Proposition is the main result of this section. 
	\begin{prop}\label{prop:coker}
    	  Let $p$ be an odd prime, and $n>p$  an integer such that $p\mid n$. Then $\nabla$ on $\Lambda_n$ satisfies 
    	 $|\coker\nabla^{k}_{(p)}|\leq p$ for $k=p+2$, $p+3$.
    	 Moreover, if $p\geq5$, then $|\coker\nabla^{k}_{(p)}|\leq p^2$ for $k=p+4$, $p+5$, and
    	 \[|\coker\nabla^{p+6}_{(p)}|\leq\begin{cases}
    	 	p^5,&\text{if }p=5,\\
    	 	p^4,&\text{if }p>5.
    	 \end{cases}\]
    \end{prop}
    In Section \ref{sec:proof}, we will see that the inequalities in Propostion \ref{prop:coker} are all equalities and the cokernels are all $\Zz/p$-vector spaces.
    
    \begin{rem}\label{rem:coker}
    For an arbitrary prime $p$, the group $\coker\nabla^{k}_{(p)}$ was computed in \cite{Gu21} for $1\leq k\leq p$, and in \cite{GZZZ22} for $k=p+1$.
    It was shown that $\coker\nabla^{k}_{(p)}=0$ for $2\leq k\leq p$ and $\coker\nabla^{p+1}_{(p)}=\Zz/p$, $\coker\nabla^1_{(p)}=\Zz/{p^r}$, where $r$ is the number defined in Theorem \ref{thm:p-torsion}. For $k=p+2$, $p+3$, the result in Propostion \ref{prop:coker}  is also proved in \cite{ZZ24}.
    \end{rem}

    The proof of Proposition \ref{prop:coker} is based on the computation of the coefficient matrix of $\nabla$ on Schur polynomials. To consider this coefficient matrix, we use the \emph{lexicographic ordering} on partitions. That is, $\mu>\lambda$ if the first $i$ for which $\mu_i\neq \lambda_i$, has $\mu_i>\lambda_i$. 
    We also use the total ordering $\succ$ on partitions by setting $\mu\succ\lambda$ if  $\bar\mu>\bar\lambda$.
    Before going into the proof of Proposition \ref{prop:coker}, we give a preliminary lemma.

    \begin{lem}\label{lem:linear}
    	Let $p,n$ be as in Proposition \ref{prop:coker}, and let $\lambda=\{\lambda_1,\dots,\lambda_n\}$ be a partition of $p+k$, $2\leq k\leq p$, such that $\lambda_1=p$ and $\lambda_2$ is even. Then in $\im\nabla_{(p)}^{p+k+1}$, there are elements of the form
    	\[\alpha=s_{\lambda}+\sum_{\mu>\lambda}a_\mu s_\mu\ \text{ and }\ \beta=s_{\bar\lambda}+\sum_{\mu\succ\bar\lambda}b_{\mu} s_{\mu},\ a_\mu,\,b_{\mu}\in\Zz_{(p)}.\]
    \end{lem}
    \begin{proof}
    First we prove the existence of $\alpha\in\im\nabla_{(p)}^{p+k+1}$ of the given form.	Since $p$ is odd and $\lambda_2$ is even, $p>\lambda_2$ and $l=(p-\lambda_2-1)/2$ is a nonnegative integer. For $0\leq i\leq l$, define partitions:
     \[\begin{split}
    	&\lambda^i=(p-i,\lambda_2+i+1,\lambda_3,\dots,\lambda_n)\vdash(p+k+1),\\
    	&(\lambda')^i=(p-i,\lambda_2+i,\lambda_3,\dots,\lambda_n)\vdash(p+k).\end{split}\]
    By Theoorem \ref{thm:nabla}, we have 
    \[\nabla(s_{\lambda^i})=(\sum_{\substack{\mu\vdash (p+k),\,\mu\subset\lambda^i\\
    			\mu_1=p-i,\,\mu_2=\lambda_2+i+1}}c_\mu s_\mu)+
\begin{cases}
	(n+a_i)s_{(\lambda')^{i+1}}+(n+b_i)s_{(\lambda')^i}&\text{if } 0\leq i<l,\\
	(n+b_i)s_{(\lambda')^{i}}&\text{if } i=l,
\end{cases}\] 
    in which $c_\mu\in\Zz$, $a_i=p-i-1$ and $b_i=\lambda_2+i-1$. Since $p\mid n$ and $0\leq i\leq l<p-\lambda_2<p-1$,  it follows that $(n+a_i)$ and $(n+b_i)$ are invertible in $\Zz_{(p)}$ for all $i$. Hence by easy linear algebra, we can write
    \begin{equation}\label{eq:mu}
    	s_\lambda=s_{(\lambda')^0}=\sum_{i=0}^l\big[k_i\nabla(s_{\lambda^i})+\sum_{\substack{\mu\vdash (p+k),\,\mu\subset\lambda^i\\
    		\mu_1=p-i,\,\mu_2=\lambda_2+i+1}}k_\mu s_\mu\big],\ k_i,\,k_\mu\in\Zz_{(p)}.
    	\end{equation}
Note that $\mu>\lambda$ for each $\mu$ with $\mu_1=p$ in \eqref{eq:mu}. 
For each $i\geq 1$, let 
\[U_i=\{\rho\vdash (p+k): \rho_1<p,\,\rho_1+\rho_2=p+\lambda_2+i,\,
\rho_j\leq \lambda_j\text{ for }j\geq 3\}.\]
Then \eqref{eq:mu} can be rewritten as 
 \begin{equation}\label{eq:rho}
	s_\lambda=\alpha'+\sum_{\mu>\lambda}k_\mu s_\mu+\sum_{\rho\in U_1}l_\rho s_\rho,\ \ \alpha'\in\im\nabla_{(p)},\ k_\mu,\,l_\rho\in\Zz_{(p)}.
\end{equation}

We will show that \eqref{eq:rho} holds when $U_1$ is replaced by $U_i$ for any $i>1$. Then the lemma follows since $U_i=\emptyset$ for $i>k-\lambda_2$. We prove this by induction on $i$. 
Let $U_i'=\{\rho\in U_i:\rho_1=p-1\}$, $i\geq 1$. For a partition $\rho\in U_i'$, define partitions:
\begin{gather*}\rho^j=(\rho_1-j+1,\rho_2+j,\rho_3,\dots,\rho_n)\vdash(p+k+1),\ \ 0\leq j\leq \frac{\rho_1-\rho_2}{2},\\
	(\rho')^j=(\rho_1-j+1,\rho_2+j-1,\rho_3,\dots,\rho_n)\vdash(p+k),\ \ 0\leq j\leq \frac{\rho_1-\rho_2}{2}+1.
	\end{gather*}
	It is easy to see that $(\rho')^0>\lambda$, $(\rho')^1=\rho$, and \[U_i=\bigcup_{\rho\in U_i'}\{(\rho')^j:1\leq j\leq \frac{\rho_1-\rho_2}{2}+1\}.\] 
 Since 
$\rho_1-j+1>\rho_2+j>\lambda_2\geq \lambda_3\geq\rho_3$ for $0\leq j\leq \frac{\rho_1-\rho_2}{2}$, 
Theorem \ref{thm:nabla} gives
\[\nabla(s_{\rho^j})=(n+a_j)s_{(\rho')^{j+1}}+(n+b_j)s_{(\rho')^j}+\sum_{\substack{\mu\vdash (p+k),\,\mu\subset\rho^j\\
		\mu_1=\rho_1-j+1,\,\mu_2=\rho_2+j}}c_\mu s_\mu,\]
where $a_j=\rho_1-j$, $b_j=\rho_2+j-2$, $c_\mu\in\Zz$.
Since $p\mid n$ and $p>\rho_1\geq \rho_2\geq\lambda_2+2>2$, $(n+a_j)$ and $(n+b_j)$ are invertible in $\Zz_{(p)}$ for all $0\leq j\leq \frac{\rho_1-\rho_2}{2}$. Hence a linear algebra argument shows that adding a proper $\Zz_{(p)}$-linear combination of $\nabla(s_{\rho^j})$, $0\leq j\leq \frac{\rho_1-\rho_2}{2}$, to the right side of the $U_i$ form of \eqref{eq:rho} makes the coefficients $l_{(\rho')^j}$ become zero for $1\leq j\leq \frac{\rho_1-\rho_2}{2}+1$, and creats a new remainder of the form
\[\sum_{j=0}^{\lfloor\frac{\rho_1-\rho_2}{2}\rfloor}\sum_{\substack{\mu\vdash (p+k),\,\mu\subset\rho^j,\\
		\mu_1=\rho_1-j+1,\,\mu_2=\rho_2+j}}l_\mu s_\mu,\ l_\mu\in\Zz_{(p)}.\]
Repeating this procedure for all $\rho\in U_i'$, we can rewrite the $U_i$ form of \eqref{eq:rho} as
\[s_\lambda=\alpha'+\sum_{\mu>\lambda}k_\mu s_\mu+\sum_{\rho\in U_{i+1}}l_\rho s_\rho,\  \  \alpha'\in\im\nabla_{(p)},\ k_\mu,\,l_\rho\in\Zz_{(p)},\]
finishing the induction step.
	This proves the existence of $\alpha$. 
	
	The proof of the existence of $\beta$ is very similar by replacing all  partitions we constructed above with their conjugates, and replacing all coefficients $(n+a_j)$ and $(n+b_j)$ in the proof of the existence of $\alpha$ with $(n-a_j)$ and $(n-b_j)$ respectively. We omit the details here.
\end{proof}
    
\begin{proof}[Proof of Proposition \ref{prop:coker}]
	 For each partition $\lambda=(\lambda_1,\dots,\lambda_n)\in\MM_n^{k-1}$ (see Theorem \ref{thm:schur} for the definition of $\MM_n^{k-1}$), let $\lambda'=(\lambda_1',\dots,\lambda_n')$ be the partition of $k$ such that $\lambda'_1=\lambda_1+1$ and $\lambda_i'=\lambda_i$ for $i\neq 1$. Then by Theorem \ref{thm:nabla} and the rule of the lexicographic ordering of partitions, we have
	\begin{equation}\label{eq:lambda'}
		\nabla(s_{\lambda'})=(n+\lambda_1)s_\lambda+\sum_{\mu>\lambda}a_\mu s_\mu,\ a_\mu\in\Zz.
	\end{equation}
	Let $M_n^k$ be the submodule of $\Lambda_n^{k}$ spanned by $\{s_{\lambda'}:\lambda\in\MM_n^{k-1}\}$.  Then the coefficient matrix associated to $\nabla:M_n^k\to \Lambda^{k-1}_n$ has the form
	\[   
	A=\bordermatrix{
		~&\lambda'&>&\cdots&\mu'\cr
		\lambda&a_1&*&\cdots&*\cr
		>&0&a_2&\cdots&*\cr
		\vdots&\vdots&\vdots&\ddots&\vdots\cr
		\mu&0&0&\cdots&a_m
	},\ m=\#\MM_n^{k-1},\]
	where $a_i$'s are of the form $n+\lambda_1$ for $\lambda\in\MM_n^{k-1}$. Therefore, $\coker\nabla^{k}$ is a finite group. 

For each $k$ in the range $p+2\leq k\leq p+6$, define
	\[\Ss_k=\{\lambda\vdash(k-1):\lambda_1=p\}\ \text{ and }\ \Ss'_k=\{\lambda\in\Ss_k:\lambda_2\text{ is even}\}.\] 
 First we consider the range $p+2\leq k\leq \min\{2p,p+6\}$, which covers all situations except for $p=5$ and $k=p+6$.
In this range any partition $\lambda\vdash(k-1)$  satisfies $\lambda_1\leq k-1<2p$. Hence \eqref{eq:lambda'} implies that a diagonal entry $a_i$ in $A$ is not invertible in $\Zz_{(p)}$ if and only if the corresponding partition $\lambda\vdash (k-1)\in\Ss_k$, so that $a_i=n+p$. For any $\lambda\vdash(k-1)\in\Ss_k'$, Lemma \ref{lem:linear} shows that there is $f\in\Lambda_n^k$ such that $\nabla_{(p)}(f)=s_{\lambda}+\sum_{\mu>\lambda}a_\mu s_\mu$. Hence, by rechoosing the basis of $M_n^k$, the diagonal entries $a_i=n+p$ corresponding to $\lambda\in\Ss_k'$ can be replaced by $1$ for $\nabla_{(p)}^k$. This shows that the order $l$ of $\coker\nabla^k_{(p)}$ satisfies $l\mid (n+p)^{r_k}$, where $r_k=\#\Ss_k-\#\Ss_k'$. It is easy to check that for $p\geq 3$,
\[\begin{split}&\#\Ss_{p+2}=\PP(1)=1,\ \ \#\Ss_{p+2}'=0,\ \ r_{p+2}=1,\\
	&\#\Ss_{p+3}=\PP(2)=2,\ \ \#\Ss_{p+3}'=1,\ \ r_{p+3}=1,
\end{split}\]
 and for $p\geq 5$,
\[\begin{split}
	&\#\Ss_{p+4}=\PP(3)=3,\ \ \#\Ss_{p+4}'=1,\ \ r_{p+4}=2,\\
	&\#\Ss_{p+5}=\PP(4)=5,\ \ \#\Ss_{p+5}'=3,\ \ r_{p+5}=2,\\
	&\#\Ss_{p+6}=\PP(5)=7,\ \ \#\Ss_{p+6}'=3,\ \ r_{p+6}=4.
\end{split}\]
So if $\frac{n}{p}\not\equiv p-1$ mod $p$, $|\coker\nabla^{k}_{(p)}|\leq p^{r_k}$. 
On the other hand, if $\frac{n}{p}\equiv p-1$ mod $p$, we use the conjugate $\bar{\lambda'}$ of $\lambda'$ for each $\lambda\in\MM_n^{k-1}$. The assumptions on $p$, $n$ and $k$ guarantee $\#\bar{\lambda'}\leq n$, so that $s_{\bar{\lambda'}}\subset \Lambda_n$. Theorem \ref{thm:nabla} gives
\[\nabla(s_{\bar{\lambda'}})=(n-\lambda_1)s_{\bar\lambda}+\sum_{\mu\succ\bar\lambda}a_\mu s_\mu,\ a_\mu\in\Zz.\] 
Let $\bar M_n^k$ be the submodule of $\Lambda_n^{k}$ spanned by the Schur polynomials $s_{\bar{\lambda'}}$. Then the coefficient matrix associated to $\nabla:\bar M_n^k\to \Lambda^{k-1}_n$ takes the form
\[   
\bar A=\bordermatrix{
	~&\bar\lambda'&\succ&\cdots&\bar\mu'\cr
\bar\lambda&\bar a_1&*&\cdots&*\cr
	\succ&0&\bar a_2&\cdots&*\cr
	\vdots&\vdots&\vdots&\ddots&\vdots\cr
	\bar\mu&0&0&\cdots&\bar a_m
},\]
where $\bar a_i$'s are of the form $n-\lambda_1$ for $\lambda\in\MM_n^{k-1}$. Applying Lemma \ref{lem:linear} to $\lambda\in\Ss_k'$ and using a similar argument as before, we can get $l\mid (n-p)^{r_k}$. Since $\frac{n}{p}-1\equiv p-2$ mod $p$ and $p\geq 3$, $p^2\nmid (n-p)$. Therefore we still have $|\coker\nabla^{k}_{(p)}|\leq p^{r_k}$,  which proves the proposition except for $p=5$ and $k=p+6$.

For $p=5$ and $k=p+6$, it is easy to see that a diagonal entry $a_i$ in $A$ is not invertible in $\Zz_{(p)}$ if and only if the corresponding partition $\lambda\vdash (p+5)=(2p)$ satisfies that
$\lambda\in\Ss_{p+6}$
or $\lambda=(2p)$, since $p$ also divides the coefficient $n+2p$ in \eqref{eq:lambda'} for $\lambda=(2p)$. 
If $\frac{n}{p}\not\equiv p-1,\,p-2$ mod $p$, then $p^2\nmid (n+p),\,(n+2p)$, and by similar reasoning we get 
$|\coker\nabla^{p+6}_{(p)}|\leq	p^{r_{p+6}+1}$.
If $\frac{n}{p}\equiv p-1$ or $p-2$ mod $p$, using a conjugate partition argument as before, we still get $|\coker\nabla^{p+6}_{(p)}|\leq p^{r_{p+6}+1}$, completing the proof.
    \end{proof}

\section{More about $\nabla$ on $\Lambda_{3k}$ and $\Lambda_{2k}$} 
In this section, we give several lemmas, which will be used in our computation of the spectral sequence $^UE$. They are supplements to Proposition \ref{prop:coker}.

\begin{lem}\label{lem:p=3(7)}
	For $\nabla$ acting on $\Lambda_{3k}$, $k\geq2$, we have 
	\[|\coker\nabla^7_{(3)}|\leq\begin{cases}
		3^3&\text{if }3\mid k,\\
		3^4&\text{if }3\nmid k.
		\end{cases}\] 
	Moreover, $\coker\nabla^7_{\Zz/3}\subset(\Zz/3)^3$ holds for all $k\geq 2$. 
\end{lem}

\begin{lem}\label{lem:p=3(8)}
	For $\nabla$ acting on $\Lambda_{3k}$, $k\geq2$, we have \[|\coker\nabla^8_{(3)}|\leq 3^3,\ \ |\coker\nabla^9_{(3)}|\leq 3^5.\]
\end{lem}  
In Section \ref{sec:proof} we will see that the cokernels in the lemmas are acturally 
\[\coker\nabla^7_{(3)}=\begin{cases}
(\Zz/3)^3&\text{if }3\mid k,\\
(\Zz/3)^2\oplus\Zz/9&\text{if }3\nmid k,
\end{cases}\]
and $\coker\nabla^8_{(3)}=(\Zz/3)^3$, $\coker\nabla^9_{(3)}=(\Zz/3)^5$.

Before proving these lemmas, let us introduce a few notations. Recall that in the proof of Proposition \ref{prop:coker}, we constructed a module $M_n^k\subset\Lambda_n^k$, which has a basis $\{s_{\lambda'}:\lambda\in\MM_n^{k-1}\}$, where $\lambda_1'=\lambda_1+1$ and $\lambda_i'=\lambda_i$ for $i\neq 1$. Let $N_n^{k}=\Lambda_n^{k}/M_n^{k}$, $\bar s_\lambda\in N_n^{k}$ the coset $s_\lambda M^k_n$ for  $\lambda\in\MM_n^{k}$.
Then there is a commutative diagram with exact rows,
\[
\xymatrix{
	0\ar[r]&M_n^{k}\ar[r]\ar[d]^{\nabla^{k}_M}&\Lambda_n^{k}\ar[r]\ar[d]^{\nabla^{k}}&N_n^{k}\ar[r]\ar[d]&0\\
	0\ar[r]&\Lambda_n^{k-1}\ar[r]&\Lambda_n^{k-1}\ar[r]& 0\ar[r]&0}
\]
where the map $\nabla_M^{k}$ is the restriction of $\nabla^{k}$ to $M_n^{k}$.
By snake lemma, there is an exact sequence 
\begin{equation}\label{eq:snake}
0\to\ker\nabla^{k}_M\to\ker\nabla^{k}\to N_n^{k}\xrightarrow{\delta}\coker\nabla^{k}_M\to\coker\nabla^{k}\to0,
\end{equation}
where the connecting homomorphism $\delta$ is induced by $\nabla$.

\begin{proof}[Proof of Lemma \ref{lem:p=3(7)}]
We have already seen in the proof of Proposition \ref{prop:coker} that the coefficient matrix associated to $\nabla^7_M:M_{3k}^7\to \Lambda^{6}_{3k}$ has the form
\[ A=  
\bordermatrix{
	~&\lambda'&>&\cdots&\mu'\cr
	\lambda&a_1&*&\cdots&*\cr
	>&0&a_2&\cdots&*\cr
	\vdots&\vdots&\vdots&\ddots&\vdots\cr
	\mu&0&0&\cdots&a_m
},\ m=\#\MM_{3k}^6,\]
where the diagonal entries $a_i$'s are of the form $3k+\lambda_1$ for $\lambda\in\MM_{3k}^6$. Hence $3\mid a_i$ if and only if $\lambda_1=3$ or $6$. There are only four such partitions: $(3,3)$, $(3,2,1)$, $(3,1,1,1)$, $(6)$. If in addition $3\mid k$, then $3^2\nmid a_i$ for all $i$, and we have $|(\coker\nabla^{7}_M)_{(3)}|=3^4$.  Furthermore, in the exact sequence \eqref{eq:snake}, we have
\begin{equation}\label{eq:coefficient}
\delta(\bar s_{(3,3,1)})=(3k-2)s_{(3,3)}+(3k+1)s_{(3,2,1)}.
\end{equation}
This means that $\delta(\bar s_{(3,3,1)})\neq 0$ in $(\coker\nabla^{7}_M)_{(3)}$, since the diagonal entries in $A$ corresponding to $(3,3)$ and $(3,2,1)$ are divided by $3$ but the coefficients of $s_{(3,3)}$ and $s_{(3,2,1)}$ in \eqref{eq:coefficient} are not. Therefore, if $3\mid k$, the order of $\coker\nabla^{7}_{(3)}$ is at most $3^3$ by \eqref{eq:snake}.
 
For $3\nmid k$, we first assume that $k\equiv 1$ mod $3$. In this case, a diagonal entry in $A$ is divided by $3^2$ only if the corresponding partition is $\lambda=(6)$. So a similar argument shows that $|(\coker\nabla^{7}_M)_{(3)}|=3^5$ and $|\coker\nabla^{7}_{(3)}|\leq 3^4$. For the remaining case that $k\equiv 2$ mod $3$, we consider the graded homomorphism 
\[\begin{split}
	\Delta:\Lambda_{6k}=\Zz[\sigma_1,\dots,\sigma_{6k}]&\to\Lambda_{3k}=\Zz[\sigma_1',\dots,\sigma_{3k}']\\
	(1+\sigma_1+\cdots+\sigma_{6k})&\mapsto (1+\sigma_1'+\cdots+\sigma_{3k}')^2.
	\end{split}\]
Here $\sigma'_i$ means the $i$th elementary symmetric polynomial in $3k$ variables, to distinguish it from the one in $6k$ variables.
It can be shown that $\Delta$ commutes with $\nabla$, and  $\Delta$, restricted to $\Zz[\sigma_1,\dots,\sigma_{3k}]$ and localized at $3$, is an isomorphism (see the discussion following Remark \ref{rem:equaltiy}). Hence $\Delta$ induces a surjection on $\coker\nabla_{(3)}$, and we reduce to the previous case since $2k\equiv 1$ mod $3$.

To prove the second statement, we consider the coefficient matrix (with entries in $\Zz/3$) for $\nabla^7_{\Zz/3}:\Lambda_{3k}^7\otimes{\Zz/3}\to \Lambda^{6}_{3k}\otimes{\Zz/3}$. This matrix is independent of $k$ when $k\geq 3$, so the cokernel of $\nabla^7_{\Zz/3}$ are all the same for $k\geq 3$. In the case that $k=2$, the only difference is that $s_{\bar{(7)}}\not\in\Lambda_6^7$. However, since $\nabla_{\Zz/3}(s_{\bar{(7)}})=0$ for $k\geq 3$, it has no effect on the group $\coker\nabla^7_{\Zz_3}$. Thus the second statement follows from the first statement.
\end{proof}

\begin{proof}[Proof of Lemma \ref{lem:p=3(8)}]
The strategy is the same as the proof of Lemma \ref{lem:p=3(7)}. First, we deal with the case that $3\mid k$. We know that the coefficient matrices for  $\nabla^j_M:M_{3k}^j\to \Lambda^{j-1}_{3k}$ are upper triangular for all $j>0$. If $j\leq 9$, for any diagonal entry $a_i=3k+\lambda_1$ corresponding to a partition $\lambda\in\MM_{3k}^{j-1}$ in the coefficient matrix, we have $3^2\nmid a_i$ since $3\mid k$. 

For $j=8$, the diagonal entry is divided by $3$ if and only if it corresponds to one of the five partitions $(6,1)$, $(3,3,1)$, $(3,2,2)$, $(3,2,1,1)$, $(3,1,1,1,1)$ in $\MM^7_{3k}$. It follows that $|(\coker\nabla^{8}_M)_{(3)}|=3^5$.
In the exact sequence \eqref{eq:snake}, we have
\[
\begin{split}
	\delta(\bar s_{(3,3,2)})&=(3k-1)s_{(3,3,1)}+(3k+1)s_{(3,2,2)},\\
	\delta(\bar s_{(3,3,1,1)})&=(3k-3)s_{(3,3,1)}+(3k+1)s_{(3,2,1,1)}.
	\end{split}
	\]
Since $3\nmid(3k+1)$, $\delta(\bar s_{(3,3,2)})$ and $\delta(\bar s_{(3,3,1,1)})$ are linearly independent over $\Zz/3$ in $(\coker\nabla^{8}_M)\otimes\Zz/3$. It follows that  the order of the  subgroup of $(\coker\nabla^{8}_M)_{(3)}$ generated by $\delta(\bar s_{(3,3,2)})$ and $\delta(\bar s_{(3,3,1,1)})$ is at least $3^2$, which proves the lemma for $\coker\nabla^{8}_{(3)}$.

For $j=9$, the diagonal entry is divided by $3$ if and only if it corresponds to one of the seven partitions $(6,2)$, $(6,1,1)$, $(3,3,2)$, $(3,3,1,1)$, $(3,2,2,1)$, $(3,2,1,1,1)$ $(3,1,1,1,1,1)$ in $\MM^8_{3k}$. Hence $|(\coker\nabla^{9}_M)_{(3)}|=3^7$. By a similar computation, we see that $\delta(\bar s_{(3,3,2,1)})$ and $\delta(\bar s_{(3,3,1,1,1)})$ are linearly independent in $(\coker\nabla^{9}_M)\otimes \Zz/3$, and we get the desired conclusion for $\coker\nabla^{9}_{(3)}$.

Next we consider the case $3\nmid k$. For simplicity we assume that $k\equiv 1$ mod $3$, since the case that $k\equiv 2$ mod $3$ can be reduced to this case as we saw in the proof of Lemma \ref{lem:p=3(7)}. The only difference of this case from $3\mid k$ is that when $\lambda=(6,1)\in\MM_{3k}^7$ (resp. $\lambda=(6,2)$ or $(6,1,1) \in\MM_{3k}^8$), the corresponding diagonal entry $3k+\lambda_1=3k+6$ is divided by $3^2$. However, we can change the basis of $M^{8}_{3k}$ (resp. $M^9_{3k}$) so that the coefficient matrix for $\nabla^8_M$ (resp. $\nabla^9_M$) is still upper triangular and satisfies the condition that $3^2\nmid a_i$ for any diagonal entry $a_i$. Then we reduce to the previous case $3\mid k$.

To do this, it suffices to show that $3s_\lambda\in\im\nabla_{(3)}$ for $\lambda=(6,1)$, $(6,2)$, $(6,1,1)$ since the desired basis can obtained by just replacing $s_{\lambda'}$ with a polynomial $f\in\Lambda_{3k}$ such that $\nabla_{(3)}(f)=3s_\lambda$. An easy calculation shows that 
\[3k\cdot s_{(6,1)}=\nabla_{(3)}(s_{(6,2)}-\frac{3k+5}{3k+1}s_{(5,3)}+\frac{(3k+4)(3k+5)}{(3k+1)(3k+2)}s_{(4,4)}).\]
The proof of $3s_{(6,2)},\,3s_{(6,1,1)}\in\im\nabla_{(3)}$ can be done by performing elementary column transformations on the following $8\times 8$ matrix associated to $\nabla^9$
\[   
\bordermatrix{
	~&(6,3)&(6,2,1)&(5,4)&(5,3,1)&(5,2,2)&(4,4,1)&(4,3,2)&(3,3,3)\cr
(6,2)&3k+1&3k-2&0&0&0&0&0&0\cr
(6,1,1)&0&3k&0&0&0&0&0&0\cr
(5,3)&3k+5&0&3k+2&3k-2&0&0&0&0\cr
(5,2,1)&0&3k+5&0&3k+1&3k-1&0&0&0\cr
(4,4)&0&0&3k+4&0&0&3k-2&0&0\cr
(4,3,1)&0&0&0&3k+4&0&3k+2&3k-1&0\cr
(4,2,2)&0&0&0&0&3k+4&0&3k+1&0\cr
(3,3,2)&0&0&0&0&0&0&3k+3&3k
}\]
Here we omit the details of the calculation.
\end{proof}

\begin{lem}\label{lem:cokernel p=2}
For $\nabla$ acting on $\Lambda_{2k}$, $k\geq2$, we have $\coker\nabla^4_{\Zz/2}=\Zz/2$, $\coker\nabla^5_{\Zz/2}=(\Zz/2)^3$. Moreover, 
\[\coker\nabla^5_{(2)}=\begin{cases}
\Zz/8\oplus(\Zz/2)^2, &\text{if }k\equiv2\textrm{ mod } 4,\\
\Zz/4\oplus(\Zz/2)^2, &\text{otherwise}.\\
\end{cases}\]  
\end{lem}
\begin{proof}
As we saw earlier, the coefficient matrix associated to $\nabla^4_M:M_{2k}^4\to \Lambda^3_{2k}$ is an upper triangular matrix, and a diagonal entry is divided by $2$ if and only if it corresponds to the partition $(3,1)\vdash 4$, which gives the diagonal entry $2k+2$. Hence $(\coker\nabla^4_M)\otimes \Zz/2=\Zz/2$. On the other hand, it is easy to check that $\nabla_{\Zz/2}(s_{(2,2)})=0$ and $\nabla_{\Zz/2}(s_{(1,1,1,1)})=\nabla_{\Zz/2}(s_{(2,1,1)})\in(\im\nabla^4_M)\otimes\Zz/2$. Then \eqref{eq:snake} gives $\coker\nabla^4_{\Zz/2}=(\coker\nabla^4_M)\otimes \Zz/2=\Zz/2$. 

The proof for $\coker\nabla^5_{\Zz/2}$ is similar. By Theorem \ref{thm:nabla}, the coefficient matrix associated to $\nabla^5_M:M_{2k}^5\to\Lambda_{2k}^4$  is
\[   
\bordermatrix{
	~&(5)&(4,1)&(3,2)&(3,1,1)&(2,1,1,1)\cr
	(4)&2k+4&2k-1&0&0&0\cr
	(3,1)&0&2k+3&2k&2k-2&0\cr
	(2,2)&0&0&2k+2&0&0\cr
	(2,1,1)&0&0&0&2k+2&2k-3\cr
	(1,1,1,1)&0&0&0&0&2k+1\cr
}\]
Hence $(\coker\nabla^5_M)\otimes \Zz/2=(\Zz/2)^3$. If $\lambda=(2,2,1)$ (or $\lambda=(1,1,1,1,1)$ when $k>2$), it is easy to check that  $\nabla_{\Zz/2}(s_\lambda)=0$, and then \eqref{eq:snake} gives $\coker\nabla^5_{\Zz/2}=(\coker\nabla^5_M)\otimes \Zz/2=(\Zz/2)^3$. 

Now we prove the second statement. First we consider the case $4\mid k$ or $k\equiv 2$ mod $8$. In this case $2k+2\equiv 2$ mod $4$,
 and 
 \[2k+4\equiv \begin{cases}
 	4\text{ mod } 8, &\text{ if } 4\mid k,\\
 	8\text{ mod } 16, &\text{ if } k\equiv 2\text{ mod }8.
 \end{cases}\] 
Using the above matrix, it is staightforward to show that the formula in the Lemma holds if $\coker\nabla^5_{(2)}$ is replaced by $(\coker\nabla^5_M)_{(2)}$ (the three generators of the direct summands can be chosen as $s_{(4)}$, $ks_{(3,1)}+(k+1)s_{(2,2)}$, $(k-1)s_{(3,1)}+(k+1)s_{(2,1,1)}$), and that $\nabla_{(2)}(s_\lambda)\in\im(\nabla_M^5)_{(2)}$ if $\lambda=(2,2,1)$ (or $\lambda=(1,1,1,1,1)$ when $k>2$). Hence using \eqref{eq:snake} we get the desired formula for $\coker\nabla^5_{(2)}$. 

If $k\equiv 6$ mod $8$, then $16\mid 2k+4$, and $(\coker\nabla^5_M)_{(2)}$ is no longer the desired form. However, we can use a conjugate partition argument as in the proof of Proposition \ref{prop:coker} to change the diagonal entry $2k+4$ to $2k-4$ and $2k+2$ to $2k-2$. Then the same reasoning applies, since $2k-2\equiv 2$ mod $4$ and $2k-4\equiv 8$ mod 16 in this case. We finish the proof for $2\mid k$.

If $2\nmid k$, then the diagonal entries divided by $2$ satisfy $2k+4\equiv 2$ mod $4$ and $2k+2\equiv 0$ mod $4$. We may assume that $2k+2\equiv 4$ mod $8$ since otherwise we can use a conjugate partition argument to change the diagonal entry $2k+2$ to $2k-2$ so that $2k-2\equiv 4$ mod $8$. Hence we can show as above that \[(\coker\nabla^5_M)_{(2)}=(\Zz/4)^2\oplus\Zz/2.\] 
It is easy to verify that $\nabla_{(2)}(s_\lambda)\not\in\im(\nabla_M^5)_{(2)}$, $\nabla_{(2)}(2s_\lambda)\in\im(\nabla_M^5)_{(2)}$ for $\lambda=(2,2,1)$, and  $\nabla_{(2)}(s_\lambda)\in\im(\nabla_M^5)_{(2)}$ for $\lambda=(1,1,1,1,1)$. So $\coker\nabla^5_{(2)}$ equals to $\Zz_4\oplus(\Zz/2)^2$ or $\Zz/4\oplus\Zz/4$ by \eqref{eq:snake}. The statement for $\coker\nabla^5_{\Zz/2}$ shows that $\coker\nabla^5_{(2)}=\Zz_4\oplus(\Zz/2)^2$ is the case since $(\coker\nabla_{(2)})\otimes\Zz/2=\coker\nabla_{\Zz/2}$.
\end{proof}

\section{The operators $\Gamma_{i}$ on $\Lambda_n$}\label{sec:Gamma}
We define operators $\Gamma_i$ on the polynomial ring $P_n=\Zz[x_1,\dots,x_n]$ for all $i\in\mathbb{N}$ as follows. For any polynomial $f(x_1,\dots,x_n)$ of degree $d$, let $\Gamma(f)\in P_n[y]$ be the polynomial 
      \[\Gamma(f)=f(x_1+y,x_2+y,\dots,x_n+y)=f_0+f_1y+\cdots+f_dy^d,\quad f_i\in P_n.\]
Then we define $\Gamma_i(f)=f_i$ for $0\leq i\leq d$, and $\Gamma_i(f)=0$ for $i>d$.  Obviously $\Gamma_i$ presevers symmetric polynomials, so $\Gamma_i$ restricts to $\Lambda_n$. We use the notation $\Gamma_i^{k}$ to denote 
$\Gamma_i:P_n^k\to P_n^{k-i}$ or $\Gamma_i:\Lambda_n^k\to\Lambda_n^{k-i}$. Let $\w \Gamma_i$ (resp. $\w \Gamma_i^k$) be the composition of $\Gamma_i$ (resp. $\Gamma_i^k$) with the projection $P_n\to P_{n-1}$ or $\Lambda_n\to\Lambda_{n-1}$ induced by $x_i\mapsto x_i$ for $i<n$ and $x_n\mapsto 0$.

The operators $\w\Gamma_i$ are related to the differentials $^Td_r$ in the spectral sequence $^TE$. To see this, we define elements $v_i'=v_i-v_n$, $1\leq i\leq n$, in the cohomology ring $H^*(BT^n)=\Zz[v_1,\dots,v_n]$. For $f\in \Zz[v_1,\dots,v_n]$ and $i\geq 0$, write 
\[\w\Gamma_i(f)=\sum_{t_1,\dots,t_n}k_{i,{t_1,\dots,t_{n-1}}}v_1^{t_1}\cdots v_{n-1}^{t_{n-1}},\quad k_{i,t_1,\dots,t_{n-1}}\in\Zz.\]
Then from Proposition \ref{prop:T} and the definition of $\w\Gamma_i$ we see that for $f(v_1,\dots,v_n)=f(v_1'+v_n,\dots,v_n'+v_n)\in H^*(BT^n)$ and $\xi\in H^*(K(\Zz,3))$, 
\begin{equation}\label{eq:differential}
^Td_r(f(v_1,\dots,v_n)\xi)=\sum_{i}\sum_{t_1,\dots,t_{n-1}}{^Kd_r}(k_{i,{t_1,\dots,t_{n-1}}}\xi v_n^i)(v_1')^{t_1}\cdots (v_{n-1}')^{t_{n-1}}.
\end{equation}

For a prime $p$, let $_p\Gamma_i$, $_p\Gamma_i^k$, $_p\w\Gamma_i$, $_p\w\Gamma_i^k$ denote the composition of $\Gamma_i$, $\Gamma_i^k$, $\w\Gamma_i$, $\w\Gamma_i^k$ with the mod $p$ reduction map $\otimes\Zz/p$, respectively.  
 
\begin{prop}\label{prop:Im}
	Let $p$ be an odd prime, and $n>p$  a positive integer such that $p\mid n$. 
	Then for any $1\leq k\leq p+2$, the image of $_p\w\Gamma_{p-1}^{p+k}$ on $\Lambda_n$ contains a $\Zz/p$-vector space of dimension  $\PP(k+1)-\PP(k)+r(k-p)$, where $r(i)=0$ for $i<0$ and $r(i)=1$ for $0\leq i\leq 2$. 
\end{prop}

\begin{rem}\label{rem:number}
For any nonnegative integer $k$, the number $\PP(k+1)-\PP(k)$ counts the partitions $\lambda\vdash(k+1)$ such that $\lambda_i\neq 1$ for all $i$. To see this, note that there is a bijection:
	\begin{align*}
		\{\text{partitions of } k+1 \text{ with a } 1\}&\leftrightarrow\{\text{partitions of } k\}\\
		(\lambda_1,\dots,\lambda_{n-1},\lambda_n=1)&\leftrightarrow (\lambda_1,\dots,\lambda_{n-1}).
	\end{align*}
\end{rem}  

\begin{rem}\label{rem:gamma}
	The result \cite[Lemma 3.2]{GZZZ22} implies that $_p\w\Gamma_{p-1}^p:\Lambda_n^{p}\to\Lambda_{n-1}^1\otimes\Zz/p\cong \Zz/p$ is surjetive.
\end{rem}

The proof of Proposition \ref{prop:Im} uses the following fact about Schur polynomials.
\begin{lem}\label{lem:terms}
Let $\lambda=(\lambda_1,\dots,\lambda_n)$ be a partition. Suppose that $\lambda'=(\lambda_1,\dots,\lambda_i)$, $\lambda''=(\lambda_{i+1},\dots,\lambda_n)$  for some $1\leq i\leq n$, and  
$l=\lambda_{i+1}+\cdots+\lambda_n$. Then in $\Lambda_n$, we have 
\begin{align*}\Gamma_l(s_{\lambda})&=r(\lambda'')\cdot s_{\lambda'}+\sum_{\mu<\lambda'}k_{\mu}s_\mu,\ k_{\mu}\in\Zz,\\
r(\lambda'')&=\#\{\text{monomial terms in }s_{\lambda''}(x_1,\dots,x_{n-i})\}.
\end{align*}
\end{lem}
Before proving Lemma \ref{lem:terms}, we recall some terminologies in polynomial theory. 
A \emph{term order} on the polynomial ring $P_n=\Zz[x_1,\dots,x_n]$ is a total order on the monomials of $P_n$ that is multiplicative, meaning that $u<v$ if and only if $uw<vw$ for any monomial $w$.
 Here we use the \emph{lexicographical term order} with $x_1>\dots>x_n$. The \emph{exponent vector} of a monomial $x_1^{a_1}\cdots x_n^{a_n}$ is the vector $\mathbf{a}=(a_1,\dots,a_n)\in\Nn^n$.  
Given a polynomial $f=\sum_{\mathbf{a}\in\Nn^n}c_{\mathbf{a}}\xx^{\mathbf{a}}$, $c_{\mathbf{a}}\in\Zz$, the \emph{initial term} $\mathrm{in}(f)=c_{\mathbf{a}}\xx^{\mathbf{a}}$ of $f$ is determined by the monomial $\xx^{\mathbf{a}}$ that is largest under the term order among those whose coefficients are nonzero in $f$. From the construction rules of Schur polynomials and semistandard tableaux, it is easy to see that for a partition $\lambda=(\lambda_1,\dots,\lambda_n)$ and the Schur polynomial $s_\lambda\in\Lambda_n$, $\mathrm{in}(s_\lambda)=x_1^{\lambda_1}\cdots x_n^{\lambda_n}$.
\begin{proof}[Proof of Lemma \ref{lem:terms}]
From the definition of  $\Gamma_l$ and the above discussion of $\mathrm{in}(s_\lambda)$, it is easily verified that the exponent vector of $\mathrm{in}(\Gamma_l(s_\lambda))$ is $(\lambda_1,\dots,\lambda_i,0,\dots,0)$. So the expression of $\Gamma_l(s_\lambda)$ in terms of a linear combination of Schur polynomials takes the form $\Gamma_l(s_\lambda)=\sum_{\mu\leq\lambda'}k_{\mu}s_\mu,\,k_{\mu}\in\Zz$. It remains to show that $k_{\lambda'}=r(\lambda'')$. 

Indeed, the coefficient $k_{\lambda'}$ is the number of monomials in $s_\lambda$, whose exponent vectors take the form $(\lambda_1,\dots,\lambda_i,\mu_{i+1},\dots,\mu_{n})$. Since such a monomial corresponds to a semistandard tableau of shape $\lambda$, it follows that the monomial $x_{i+1}^{\mu_{i+1}}\cdots x_n^{\mu_n}$ corresponds to a semistandard tableau, using the numbers from $i+1$ to $n$, of shape $\lambda''$. Clearly the number $r(\lambda'')$ counts these monomials.
\end{proof}

\begin{proof}[Proof of Proposition \ref{prop:Im}]
Let $\Ss=\{\lambda\vdash k+1:\lambda_i> 1 \text{ for all } 1\leq i\leq \#\lambda\}$. Then $\lambda$ has at most $\lfloor\frac{k+1}{2}\rfloor$ parts (rows). 
For each  $\lambda\in\Ss$, define 
\[\lambda'=\begin{cases}
(\lambda_1,1,\dots,1)\vdash k+p&\text{ if }\#\lambda=1,\\
(\lambda_1,\dots,\lambda_m,2,\dots,2)\vdash k+p&\text{ if }\#\lambda=m>1.
\end{cases}\]
Recall that $p$ is an odd prime, so $(2,\dots,2)\vdash p-1$ above is well-defined. Let $\Ss_0=\{\lambda\in\Ss:\#\lambda\leq(p+1)/2\}$.
Note that if $1\leq k\leq p+1$, then $\Ss_0=\Ss$; if $k=p+2$, the only element in $\Ss\setminus\Ss_0$ is $(2,\dots,2)$ with $\frac{(p+3)}{2}$ $2$'s.

Suppose that $\lambda\in\Ss_0$. From Lemma \ref{lem:terms} we see that $\Gamma_{p-1}(s_{\lambda'})=r(\lambda'') s_\lambda+\sum_{\mu<\lambda}k_{\mu}s_\mu$, where 
\[\lambda''=\begin{cases}
	(1,\dots,1)\vdash (p-1)&\text{ if }\#\lambda=1,\\
	(2,\dots,2)\vdash (p-1)&\text{ if } \#\lambda>1,
\end{cases}\]
and $r(\lambda'')$ is by definition  the number of monomial terms in $s_{\lambda''}(x_1,\dots,x_{n-\#\lambda})$.
Hence, if $\#\lambda=1$, then $r(\lambda'')=\binom{n-1}{p-1}$, so that $p\nmid r(\lambda'')$. 
On the other hand, if $\#\lambda>1$, then for any box $(i,j)\in\lambda''$, we have $1\leq i\leq (p-1)/2,\,j=1,2$. Therefore, $p\nmid (n-\#\lambda-i+j)$, and we still have $p\nmid r(\lambda'')$ by Theorem \ref{thm:hook}. 
It follows that the elements $_p\Gamma_{p-1}^{k+p}(s_{\lambda'})$ for $\lambda\in\Ss_0$ are linearly independent over $\Zz/p$, and span a $\Zz/p$-vector space of dimension 
\begin{equation}\label{eq:S_0}
\#\Ss_0=\begin{cases}
	\#\Ss=\PP(k+1)-\PP(k)&\text{ if }1\leq k\leq p+1,\\
	\#\Ss-1=\PP(k+1)-\PP(k)-1&\text{ if } k=p+2.
\end{cases}
\end{equation}
(See Remark \ref{rem:number}.) Furthermore, since $\#\lambda<p<n$ for $\lambda\in\Ss_0$, $s_\lambda(x_1,\dots,x_{n-1})\neq0$ in $\Lambda_{n-1}$. Hence the elements $\{_p\w\Gamma_{p-1}^{k+p}(s_{\lambda'}):\lambda\in\Ss_0\}$ are also linearly independent over $\Zz/p$, and the Proposition for $1\leq k<p$ is proved.

For $k\geq p$, and any $\mu=\{\mu_1,\dots,\mu_s\}\vdash (k-p)$, let $\mu'$ and $\mu''$ be the conjugates of $\{2p,\mu_1,\dots,\mu_s\}\vdash(k+p)$ and $\{p+1,\mu_1,\dots,\mu_s\}\vdash(k+1)$, respectively. Let $\rho=(1,\dots,1)\vdash (p-1)$.  By Lemma \ref{lem:terms}, we have $\Gamma_{p-1}(s_{\mu'})=r(\rho)s_{\mu''}+\sum_{\lambda<\mu''}k_{\lambda}s_\lambda$, and $r(\rho)=\binom{n-p-1}{p-1}$. Since $p\mid n$, $p\nmid r(\rho)$. 
Hence the same argument as above shows that the $\Zz/p$-vector subspace of $\im_p\w\Gamma_{p-1}^{k+p}$ spanned by $\{_p\w\Gamma_{p-1}^{k+p}(s_{\mu'}):\mu\vdash(k-p)\}$ has dimension $\PP(k-p)$, and is linearly independent with the one spanned by $\{_p\w\Gamma_{p-1}^{k+p}(s_{\lambda'}):\lambda\in\Ss_0\}$.
Then the proposition for $p\leq k\leq p+2$ follows from the equations $\PP(0)=\PP(1)=1$, $\PP(2)=2$ and \eqref{eq:S_0}. 
\end{proof}

\section{Proof of Theorem \ref{thm:p-torsion}}\label{sec:proof}
Suppose that $p$ is an odd prime. By the result of \cite{Vis07}, Theorem \ref{thm:p-torsion} holds for $n=p$, so we assume $n>p$, $p\mid n$ in this section.

We prove Theorem \ref{thm:p-torsion} by computing the Serre spectral sequence $^UE$.
Note that $_pG={_p[}G_{(p)}]$ for any abelian group $G$ and $H^*(-)_{(p)}\cong H^*(-;\Zz_{(p)})$, so we only consider the localization of $^UE$ at $p$, where the $E_2$ page becomes
\[(^UE_2^{s,t})_{(p)}=H^s(K(\Zz,3);\Zz_{(p)})\otimes H^t(BU_n).\]

For simplicity, we will use $^UE$, $^TE$ to denote the corresponding $p$-local Serre spectral sequences.

For Theorem \ref{thm:p-torsion} \eqref{item:1}, we only consider odd dimensions $2i+1$ for $p+2\leq i\leq p+6$ since the range $i<p+2$ is covered by the works of \cite{Gu21} and \cite{GZZZ22}. In this range of $i$, by Proposition \ref{prop:K_3 odd prime} and equation \eqref{eq:BPUn} we know that if $p\geq 5$, the only nontrivial entries of total degree $2i+1$  in $^UE_2^{*,*}$ are
 \begin{gather*}{^UE}^{3,2i-2}_2\cong x_1\cdot H^{2i-2}(BU_n),\ \text{ and}\\ {^UE}^{2p+5,2i-2p-4}_2\cong x_1y_{p,0}\cdot H^{2i-2p-4}(BU_n).
\end{gather*} 
This is also true if $p=3$ and $i<p+6$, and there is in addition a nontrivial entry ${^UE}^{4p+7,0}_2\cong x_1y_{p,0}^2\cdot H^0(BU_n)$ for $p=3$, $i=p+6=2p+3$.
Hence Theorem \ref{thm:p-torsion} \eqref{item:1} is a consequence of the following three lemmas.
\begin{lem}\label{lem:odd}
In the range $1\leq k\leq 5$, $^UE^{3,2p+2k}_\infty=0$ holds for $p\geq 5$ or $k\neq p=3$.
\end{lem}

\begin{lem}\label{lem:p=3}
	If $p=3$, then
	\[{^UE}^{3,4p}_\infty=\begin{cases}
		0, &\text{if }p^2\mid n,\\
		\Zz/p,&\text{if }p^2\nmid n.
	\end{cases}\]
\end{lem}

\begin{lem}\label{lem:2p+5}
In the range $0\leq k\leq p+5$,	
\[{^UE}^{2p+5,2k}_\infty={^UE}^{4p+7,2k}_\infty=\begin{cases}
\Zz/p, &k=0,\\
0,&1\leq k\leq p+5.
\end{cases}\]
\end{lem}
Note that the differentials in $^UE$ have the form
\begin{equation}\label{eq:degree change}
	d_r:{^UE}^{s,t}_r\to{^UE}^{s+r,t-r+1}_r.
\end{equation}

\begin{proof}[Proof of Lemma \ref{lem:odd}]
	By Proposition \ref{prop:K_3 odd prime} the first three nontrivial columns of $^UE^{*,*}_3$ are $^UE_3^{0,*}\cong H^*(BU_n)$, $^UE_3^{3,*}\cong x_1\cdot H^*(BU_n)$ and $^UE_3^{2p+2,*}\cong y_{p,0}\cdot H^*(BU_n)$.
	So $^UE^{3,*}_*$ appears as the target of a nontrial differential $d_r$ only if $r=3$, and the first possible nontrivial differential originating at $^UE^{3,t}_*$ is $d_{2p-1}$  for degree reasons. 
	It follows that 
	\begin{gather}
	\coker d_3^{0,*}={^UE}_4^{3,*}=\cdots={^UE}_{2p-1}^{3,*},\label{eq:E_4 odd}\\
	{^UE}_{2p}^{3,*}=\ker d_{2p-1}^{3,*}\label{eq:E_infty odd}
	\end{gather}

	Proposition  \ref{prop:T} shows that the differential $d_3^{0,*}$ is just $\nabla_{(p)}(-)\cdot x_1$, where the action of $\nabla_{(p)}$ restricts to $H^*(BU_n)_{(p)}$, viewing $H^*(BU_n)_{(p)}$ as the symmetric polynomial ring $(\Lambda_n)_{(p)}\subset H^*(BT^n)_{(p)}\cong \Zz_{(p)}[v_1,\dots,v_n]$.  \eqref{eq:E_4 odd} shows that ${^UE}_{2p-1}^{3,2k}\cong \coker\nabla^{k+1}_{(p)}$,
	so the orders of the finite groups ${^UE}_{2p-1}^{3,2p+2k}$, $1\leq k\leq 5$ ($k\neq 3$ if $p=3$), satisfy the inequalities in Proposition \ref{prop:coker} and Lemma \ref{lem:p=3(8)}. That is
	\begin{equation}\label{eq:orders of E^3}
		|{^UE}_{2p-1}^{3,2p+2k}|\leq\begin{cases}
			p,&k=1,2,\\
			p^2, &k=3,4,\,p\neq 3,\\
			p^3,&k=4,\,p=3,\\
			p^5,&k=5,\,p=3,5,\\
			p^4,&k=5,\,p>5.
		\end{cases}
		\end{equation}	
	
	On the other hand, from \eqref{eq:differential} and the fact that $\Psi^*{^Ud}_r={^Td}_r\Psi^*$ for $\Psi^*:{^UE}\to{^TE}$,
	we see that the composition 
	\[\begin{split}
D:\ &\Lambda_n\to(\Lambda_n)_{(p)}\cong {^UE}_3^{3,*}\to \coker d_3^{0,*}={^UE}_{2p-1}^{3,*}\xrightarrow{d_{2p-1}}{^UE}_{2p-1}^{2p+2,*}\\
&\xrightarrow{\Psi^*}{^TE}_{2p-1}^{2p+2,*}\hookrightarrow {^TE}_{3}^{2p+2,*}\cong\Zz/p[v_1,\dots,v_n]\to \Zz/p[v_1,\dots,v_{n-1}]
	\end{split} \]
	is the same as the composition of the map $_p\w\Gamma_{p-1}:\Lambda_n\to\Lambda_{n-1}\otimes\Zz/p$ defined in Section \ref{sec:Gamma} with the inclusion $\Lambda_{n-1}\otimes\Zz/p\to \Zz/p[v_1,\dots,v_{n-1}]$.
	 Hence by Proposition \ref{prop:Im}, for $*=2p+2k$ in ${^UE}_3^{3,*}$ above, $\im D$ contains a $\Zz/p$-vector space of dimension 
	 \[l_k=\begin{cases}
	 \PP(k+1)-\PP(k),&k<p,\\
	 \PP(k+1)-\PP(k)+1,&p\leq k\leq p+2.
	 \end{cases}\]
	 A direct computation shows that 
	 	\begin{equation}\label{eq:l_k}
	 	l_k=\begin{cases}
	 		1,&k=1,2,\\
	 		2, &k=3,4,\,p\neq 3,\\
	 		3, &k=3,4,\,p=3,\\
	 		5,&k=5,\,p=3,5,\\
	 		4,&k=5,\,p>5.
	 	\end{cases}
	 \end{equation}	
Comparing \eqref{eq:l_k} with \eqref{eq:orders of E^3} and using \eqref{eq:E_infty odd} we immediately get ${^UE}_{2p}^{3,2p+2k}=0$ for $1\leq k\leq 5$ ($k\neq 3$ if $p=3$), and the lemma follows.
\end{proof}
\begin{rem}\label{rem:equaltiy}
	The proof of Lemma \ref{lem:odd} implies that the inequalities in Propostion \ref{prop:coker} and Lemma \ref{lem:p=3(8)} are actually equalities and the cokernels there are all $(\Zz/p)$-vector spaces as we mentioned earlier.
\end{rem}

To proof the next two lemmas, we use the diagonal map from $BPU_p$ to $BPU_n$. Since $p\mid n$, there is a diagonal map of matrices 
\[U_p\to U_n,\quad A\mapsto \begin{bmatrix}
	A&0&\cdots&0\\
	0&A&\cdots&0\\
	\vdots&\vdots&\ddots&\vdots\\
	0&\cdots&\cdots&A
\end{bmatrix},\] which passes to $PU_p\to PU_n$.
These diagonal maps induce maps $\Delta:BU_p\to BU_n$ and $\bar\Delta:BPU_p\to BPU_n$, and a commutative diagram of fibrations
\begin{equation}\label{diag:diagonal}
\begin{gathered}
	\xymatrix{
	BU_p\ar[r]\ar^{\Delta}[d]&BPU_p\ar[r]\ar[d]^{\bar\Delta}&K(\Zz,3)\ar[d]^{=}\\
	BU_n\ar[r]&BPU_n\ar[r]&K(\Zz,3)}
	\end{gathered}
	\end{equation}
The sepectral sequence map on $E_2$ pages is given by the induced homomorphism   $\Delta^*:H^*(BU_n)\to H^*(BU_p)$. Since the Serre spectral sequence map is functorial, $\Delta^*$ is commutative with the operator $\nabla={^U}d_3^{0,*}$, which proves the guaranteed statement in the proof of Lemma \ref{lem:p=3(7)}.

Now let us describe $\Delta^*$ as follows. Let $c_i$ and $c_i'$ be the $i$th universal Chern classes of $BU_n$ and $BU_p$ respectively. As we know, $H^*(BU_n)=\Zz[c_1,\dots,c_n]$ and $H^*(BU_p)=\Zz[c_1'\dots,c_p']$.
Since the map $\Delta$ factors through the diagonal map $BU_p\to (BU_p)^m$, $m=n/p$, using the Whitney sum formula and the functorial property of total Chern classes we immediately get
\begin{equation}\label{eq:chern class}
	\Delta^*(1+c_1+\cdots+c_n)=(1+c_1'+\cdots+c_p')^m.
\end{equation}
Moreover, if $p\nmid m$, then $\Delta^*$ restricted to $\Zz_{(p)}[c_1,\dots,c_p]$ is an isomorphism, since $\Delta^*(\frac{c_i}{m})=c_i'+$ higher order terms for $1\leq i\leq p$ by \eqref{eq:chern class}, using the lexicographical term order on $\Zz_{(p)}[c_1'\dots,c_p']$ with respect to $c_1'>\cdots>c_n'$.

\begin{proof}[Proof of Lemma \ref{lem:p=3}]
For $p=3$, Lemma \ref{lem:p=3(7)} gives 
\[|{^UE}_{2p-1}^{3,4p}|\leq \begin{cases}
	p^3,&\text{if }p^2\mid n,\\
	p^4, &\text{if }p^2\nmid n.
	\end{cases}
	\] 
	Using \eqref{eq:l_k}  for the case that $k=p=3$ and applying the same reasoning as in the proof of Lemma \ref{lem:odd}, we see that ${^UE}_{\infty}^{3,4p}=0$ if $p^2\mid n$ and  ${^UE}_{\infty}^{3,4p}\subset\Zz/p$ if $p^2\nmid n$. 
	It remains to show that ${^UE}_{\infty}^{3,4p}\neq0$ for the second case. 
	
	By \cite[Theorem 3.4]{Vis07}, for the spectral sequence $^UE$ for $BPU_p$ ($p$ an arbitrary odd prime) there is an element $\alpha'\in {^UE}_{\infty}^{0,2p^2-2p}={^UE}_4^{0,2p^2-2p}\subset\ker\nabla_{(p)}\subset \Zz_{(p)}[c_1',\dots,c_p']$ such that $0\neq x_1\alpha'\in {^UE}_{\infty}^{3,2p^2-2p}$. If $p^2\nmid n$, then $\Delta^*$ restricted to $\Zz_{(p)}[c_1,\dots,c_p]$ is an isomorphism, it follows that for the spectral sequence $^UE$ for $BPU_n$ there is an element $\alpha\in {^UE}_{\infty}^{0,2p^2-2p}={^UE}_4^{0,2p^2-2p}\subset\ker\nabla_{(p)}\subset \Zz_{(p)}[c_1,\dots,c_p]$ such that $\bar\Delta^*(\alpha)=\alpha'$, where the equality ${^UE}_{\infty}^{0,2p^2-2p}={^UE}_4^{0,2p^2-2p}$ is from Theorem \ref{thm:E_4}. Hence $0\neq x_1\alpha\in {^UE}_{\infty}^{3,2p^2-2p}$ since $\bar\Delta^*(x_1\alpha)=x_1\alpha'\neq 0$, which finishes the proof since $4p=2p^2-2p$ for $p=3$.  
\end{proof}

\begin{rem}\label{rem:p=3}
	The proof of Lemma \ref{lem:p=3}, together with the second statement of Lemma \ref{lem:p=3(7)}, implies that $\nabla$ on $\Lambda_n$ has
	\[\coker\nabla^7_{(3)}=\begin{cases}
			(\Zz/3)^3&\text{if }9\mid n,\\
	(\Zz/3)^2\oplus\Zz/9&\text{if }3\mid n,\,9\nmid n.
	\end{cases}\] 
\end{rem}

The proof of Lemma \ref{lem:2p+5} uses the fact that $x_1y_{p,0}$, $x_1y_{p,0}^2\neq 0$ in $H^*(BPU_n)$.  This comes from the the fact that their images under the homomorphism induced by $\bar\Delta:BPU_p\to BPU_n$ is not zero in  $H^*(BPU_p)$ by \cite[Theorem 3.4]{Vis07} and Theorem \ref{thm:torsion elements}. Here the notations $x_1y_{p,0}$, $x_1y_{p,0}^2\in H^*(K(\Zz,3))$ are abused to denote their images under the homomorphism induced by the fibration $BPU_n\to K(\Zz,3)$.

\begin{proof}[Proof of Lemma \ref{lem:2p+5}]
Recall that	\[{^UE}_3^{2p+5,*}\cong x_1y_{p,0}\cdot H^*(BU_n)\ \text{ and }\ {^UE}_3^{4p+7,*}\cong x_1y_{p,0}^2\cdot H^*(BU_n)\]
are $\Zz/p$-vector spaces, and similar to the proof of Lemma \ref{lem:odd}, one can show that 
\[\begin{split}
	\coker\nabla_{\Zz/p}\cong\coker d_3^{2p+5,*}={^UE}_4^{2p+5,*}=\cdots={^UE}_{2p-1}^{2p+5,*},\ \ {^UE}_{2p}^{2p+5,*}=\ker d_{2p-1}^{2p+5,*},\\
\coker\nabla_{\Zz/p}\cong\coker	d_3^{4p+7,*}={^UE}_4^{4p+7,*}=\cdots={^UE}_{2p-1}^{4p+7,*},\ \ {^UE}_{2p}^{4p+7,*}=\ker d_{2p-1}^{4p+7,*}.
\end{split}\]
Hence, using Remark \ref{rem:equaltiy}, \ref{rem:p=3} and the results we mentioned in Remark \ref{rem:coker}, \ref{rem:gamma}, and applying an argument as in the proof of Lemma \ref{lem:odd}, it can be proved that 
\[{^UE}^{2p+5,2k}_{2p}={^UE}^{4p+7,2k}_{2p}=\begin{cases}
	\Zz/p, &k=0,\\
	0,&1\leq k\leq p+5,
\end{cases}\]
in which ${^UE}^{2p+5,0}_{2p}=\Zz/p\{x_1y_{p,0}\}$ and ${^UE}^{4p+7,0}_{2p}=\Zz/p\{x_1y_{p,0}^2\}$. 
Since $x_1y_{p,0}$ and $x_1y_{p,0}^2$ are not zero in $H^*(BPU_n)$ as we discussed above, they must survive to $E_\infty$, and the lemma follows.
\end{proof}

Now we prove Theorem \ref{thm:p-torsion} \eqref{item:2}. Note that if the total degree $s+t$ of $^UE_2^{s,t}$ is even and satisfies $0\leq s+t\leq 4p+2$, then $^UE^{s,t}_2$ could be nonzero only if $s=0$ or $2p+2$ by Proposition \ref{prop:K_3 odd prime} and equation \eqref{eq:BPUn}. Hence for $p+1\leq i\leq 2p+1$ there is an exact sequence:
\[0\to {^UE}^{2p+2,2i-2p-2}_\infty\to H^{2i}(BPU_n)_{(p)}\to{^UE}^{0,2i}_\infty\to0.\]
Moreover, since $^UE^{0,*}_\infty\subset {^UE}^{0,*}_2\cong (\Lambda_n)_{(p)}$ is a free $\Zz_{(p)}$-module, the above short exact sequence 
splits and we get 
\[_pH^{2i}(BPU_n)={^UE}^{2p+2,2i-2p-2}_\infty\quad \text{for }p+1\leq i\leq 2p+1.\]
If $s+t=$ $4p+4$ or $4p+6$, we also need to consider ${^UE}^{4p+4,*}_*$ because $^UE^{4p+4,*}_2\cong y_{p,0}^2\cdot H^*(BPU_n)$. By the theory of spectral sequences, for $i=2p+2,2p+3$, there are two exact sequences:
\begin{gather*}
	0\to {^UE}^{4p+4,2i-4p-4}_\infty\to G\to{^UE}^{2p+2,2i-2p-2}_\infty\to0,\\
	0\to G\to H^{2i}(BPU_n)_{(p)}\to{^UE}^{0,2i}_\infty\to0.
\end{gather*}
Hence the first part of Theorem \ref{thm:p-torsion} \eqref{item:2} follows from the next 
\begin{lem}\label{lem:even}
${^UE}^{2p+2,2i}_\infty={^UE}^{4p+4,2i}_\infty=\begin{cases}\Zz/p,&i=0,\\
		0,&1\leq i\leq p+3.
	\end{cases}$
\end{lem}

\begin{proof}
We only prove the result for ${^UE}^{2p+2,2i}_\infty$, since the proof for ${^UE}^{4p+4,2i}_\infty$ is very similar.
The cases $i=0,1$ were proved in \cite{GZZZ22}, so we assume $i>1$.	

From Proposition \ref{prop:K_3 odd prime} and \eqref{eq:degree change} we see that the first and second possible nontrivial differentials originating at $^UE^{2p+2,*}_*$ are $d_3$ and $d_{2p+5}$ respectively, and ${^UE}^{2p+2,*}_*$  appears as the target of a nontrial differential $d_r$ only if $r=2p-1$.
	It follows that
	\begin{gather}
	{^UE}_{2p-1}^{2p+2,*}=	{^UE}_{2p-2}^{2p+2,*}=\cdots=	{^UE}_4^{2p+2,*}=\ker d_3^{2p+2,*},\label{eq:E_4 even}\\
	{^UE}^{2p+2,*}_{2p}=\coker d_{2p-1}^{3,*+2p-2}.\label{eq:E_infty even}
	\end{gather}
	
	Since ${^UE}_3^{2p+2,*}\cong y_{p,0}\cdot H^*(BU_n)$ and ${^UE}_3^{2p+5,*}\cong x_1y_{p,0}\cdot H^*(BU_n)$ are isomorphic to $\Lambda_n\otimes\Zz/p$, and $d_3^{2p+2,*}$ can be identified with $\nabla_{\Zz/p}(-)\cdot x_1$, we have a commutative diagram of exact sequences of $\Zz/p$-vector spaces: 
	\[\xymatrix{
		0\ar[r]&\ker d_3^{2p+2,2i}\ar[r]\ar^{\cong}[d]&{^UE}_3^{2p+2,2i}\ar[r]^{d_3}\ar[d]^{\cong}&{^UE}_3^{2p+5,2i-2}\ar[r]\ar[d]^{\cong}&\coker d_3^{2p+2,2i}\ar[r]\ar[d]^{\cong}&0\\
		0\ar[r]&\ker\nabla_{\Zz/p}^i\ar[r]&\Lambda_n^i\otimes\Zz/p\ar[r]^{\nabla_{\Zz_p}}&\Lambda_n^{i-1}\otimes\Zz/p\ar[r]&\coker\nabla_{\Zz/p}^i\ar[r]&0}\]
		From this we get 
		\begin{equation}\label{eq:ker}
			\begin{split}
				\dim_{\Zz/p}\ker d_3^{2p+2,2i}&=\dim_{\Zz/p}\ker\nabla_{\Zz/p}^i\\
				&=\PP(i)-\PP(i-1)+\dim_{\Zz/p}\coker\nabla_{\Zz/p}^i,
			\end{split}
			\end{equation}
		where the second equality comes from the fact that $\dim_{\Zz/p}(\Lambda_n^k\otimes\Zz/p)=\PP(k)$.
		For $2\leq i\leq p+3$, we have
		\begin{equation}\label{eq:cases ker}
			\dim_{\Zz/p}\coker\nabla_{\Zz/p}^i=\begin{cases}
	0,&2\leq i\leq p,\\
	1,&p<i\leq p+3.
		\end{cases}\end{equation}
		Recall that \eqref{eq:cases ker} was proved in \cite{Gu21} for $2\leq i\leq p$ and in \cite{GZZZ22} for $i= p+1$, and the cases $i=p+2$, $p+3$ are given by Proposition \ref{prop:coker} and Remark \ref{rem:equaltiy}.
		
		On the other hand, using the same argument as in the proof of Lemma \ref{lem:odd}, one can see from Proposition \ref{prop:Im} that 
		\begin{equation}\label{eq:cases coker}
			\dim_{\Zz/p}\im d_{2p-1}^{3,2p+2i-2}\geq\begin{cases}
				\PP(i)-\PP(i-1),&2\leq i\leq p,\\
				\PP(i)-\PP(i-1)+1,&p<i\leq p+3.
		\end{cases}\end{equation}
	
	Finally, 
	combining \eqref{eq:E_4 even}-\eqref{eq:cases coker} gives the desired result ${^UE}^{2p+2,2i}_\infty={^UE}^{2p+2,2i}_{2p}=0$ for $2\leq i\leq p+3$.
	\end{proof}

Indeed, $_pH^{4p+8}(BPU_n)=0$ for $p>3$ also comes from Lemma \ref{lem:even}. This is because when $p>3$, the nonzero elements in $^UE^{*,*}_2$ of total degree $4p+8$ are concentrated in $^UE^{0,4p+8}_2$, $^UE^{2p+2,2p+6}_2$ and $^UE^{4p+4,4}_2$. Lemma \ref{lem:even} shows that 
${^UE}^{2p+2,2p+6}_\infty={^UE}^{4p+4,4}_\infty=0.$
So we get $_pH^{4p+8}(BPU_n)=0$ for $p>3$. 

For $p=3$, we also need to consider $^UE^{4p+8,0}_*$, since $4p+8=2p^2+2=\deg(y_{p,1})$ in this case, and then $^UE^{4p+8,0}_2\cong H^{4p+8}(K(\Zz,3))_{(p)}= \Zz/p\{y_{p,1}\}$. By Theorem \ref{thm:torsion elements}, $y_{p,1}$ is not zero in $H^*(BPU_n)$, hence $_pH^{4p+8}(BPU_n)=\Zz/p$ for $p=3$, and we finish the proof of Theorem \ref{thm:p-torsion}.

\section{Proof of Theorem \ref{thm:below 12}}\label{sec:proof below 12}
First we verify the group structure. By \cite[Theorem 1.1]{Gu21}, we only need to consider the cohomology group in dimension $11$. Similar to what we did in Section \ref{sec:proof}, the proof involves calculations of entries of total degree $11$ in the spectral sequence $^UE$, and we only need to calculate their localizations at $2$, since their localizations at a prime $p>2$ were given in Section \ref{sec:proof}. For this reason, we assume that $2\mid n$ and $n>2$ unless otherwise stated, since the case $2\nmid n$ is trivial and $H^*(BPU_2)$ is known. 

For simplicity, we will use $^UE$, $^TE$ to denote the corresponding $2$-local Serre spectral sequences in this section.

By Proposition \ref{prop:K_3 below 14}, $^UE^{s,t}_2$, $s+t=11$, could be nonzero only if $s=3$, $t=8$ or $s=9$, $t=2$, where 
\begin{gather*}
	{^UE}^{3,8}_2={^UE}^{3,8}_3\cong x_1\cdot H^{8}(BU_n)\cong\Lambda_n^4\otimes\Zz_{(2)},\\
	{^UE}^{9,2}_2={^UE}^{9,2}_3\cong x_1^3\cdot H^{2}(BU_n)\cong\Lambda_n^1\otimes\Zz/2.
\end{gather*}
We know that the differential $d_3$ can be identified with $\nabla(-)\cdot x_1$,  so there is a commutative diagram of exact sequences
\begin{equation}
	\begin{gathered}
\label{eq:diagram p=2}\xymatrix{
	{^UE}_3^{0,*}\ar[r]^{d_3^{0,*}}\ar^{\cong}[d]&	{^UE}_3^{3,*}\ar[r]^{d_3^{3,*}}\ar^{\cong}[d]&{^UE}_3^{6,*}\ar[r]^{d_3^{6,*}}\ar[d]^{\cong}&{^UE}_3^{9,*}\ar[d]^{\cong}\\
	\Lambda_n\otimes\Zz_{(2)}\ar[r]^{\nabla_{(2)}}&	\Lambda_n\otimes\Zz_{(2)}\ar[r]^{\nabla_{(2)/2}}&\Lambda_n\otimes\Zz/2\ar[r]^{\nabla_{\Zz/2}}&\Lambda_n\otimes\Zz/2}
	\end{gathered}\end{equation}
Here the middle map $\nabla_{(2)/2}$ at the buttom line is the composition of $\nabla_{(2)}$ with the reduction $\Lambda_n\otimes\Zz_{(2)}\to\Lambda_n\otimes\Zz/2$. From the above diagram we immediately get ${^UE}^{9,2}_4=0$ since ${^UE}^{9,2}_3\cong \Zz/2\{\sigma_1\}$ and $\nabla_{\Zz/2}(\sigma_2)=\sigma_1$. 
Therefore, we have a isomorphism of groups:
\[_2H^{11}(BPU_n)\cong {^U}E_\infty^{3,8}.\]

Now we compute ${^U}E_\infty^{3,8}$. Since ${^UE}_4^{3,8}=\ker d_3^{3,8}/\im d_3^{0,10}$, the diagram \eqref{eq:diagram p=2}  gives an exact sequence
\begin{equation}\label{eq:exact p=2}
0\to{^UE}_4^{3,8}\to\coker\nabla_{(2)}^5\to\Lambda_n^3\otimes\Zz/2\to\coker\nabla_{(2)/2}^4\to0.
\end{equation}
\begin{lem}\label{lem:E^{3,8}}
$|{^UE}_4^{3,8}|=\begin{cases}
	8,&\text{if }n\equiv 4\text{ mod } 8,\\
	4,&\text{otherwise}.
\end{cases}$
\end{lem}
\begin{proof}
Since $\Lambda_n^3\otimes\Zz/2\cong(\Zz/2)^3$, and $\coker\nabla_{(2)/2}^4=\coker\nabla_{\Zz/2}^4=\Zz/2$ by Lemma \ref{lem:cokernel p=2}, the exactness of \eqref{eq:exact p=2} gives a short exact sequence
\[0\to{^UE}_4^{3,8}\to\coker\nabla_{(2)}^5\to(\Zz/2)^2\to0.\]
Hence the lemma follows by Lemma \ref{lem:cokernel p=2}.
\end{proof}

If the coefficient ring $\Zz_{(2)}$ is replaced by $\Zz/2$, then we will still have ${^UE}^{9,2}_4=0$, and the $\Zz/2$ analog of \eqref{eq:exact p=2} will give ${^UE}_4^{3,8}=\Zz/2$ by Lemma \ref{lem:cokernel p=2}. This implies that $_2H^{11}(BPU_n)$ is a cyclic group since \[H^{11}(BPU_n)\otimes\Zz/2\subset H^{11}(BPU_n;\Zz/2)\subset{^UE}_4^{3,8}=\Zz/2.\]

\begin{lem}\label{lem:d_7}
	In the spectral sequence $^UE$, we have	
	
	(a) ${^U}E_\infty^{3,8}={^U}E_8^{3,8}\subset{^U}E^{3,8}_7=\cdots= {^U}E_4^{3,8}$.
	
	(b) ${^U}E^{10,2}_3=\cdots={^U}E^{10,2}_{7}=\Zz/2\{y_{2,1}c_1\}$.
	
	(c) ${^U}d_7(x_1c_2^2)=y_{2,1}c_1$ if $n\equiv 2$ mod $4$, and ${^U}d_7(2x_1c_4)=y_{2,1}c_1$ if $4\mid n$.
\end{lem}
\begin{proof}
	Part (a) follows from the fact that $d_r^{3,8}=0$ for $4\leq r\leq 6$ and $r\geq 8$, which comes from an easy check of the nontrivial columns of ${^U}E^{*,12-*}_3$ by use of Proposition \ref{prop:K_3 below 14}.  
	
	For part (b), we have ${^U}E^{10,2}_3\cong y_{2,1}\cdot H^2(BU_n)\cong \Zz/2\{y_{2,1}c_1\}$, and $d_3(y_{2,1}c_1)=ny_{2,1}=0$, it follows that ${^U}E^{10,2}_3={^U}E^{10,2}_4$. Furthermore, inspection of degrees gives ${^U}E^{10,2}_4=\cdots={^U}E^{10,2}_7$.
	
	For part (c), we first note that $\nabla_{(2)/2}(\sigma^2_2)=\nabla_{(2)/2}(2\sigma_4)=0$, so $x_1c_2^2$, $2x_1c_4\in {^U}E^{3,8}_4$ by diagram \eqref{eq:diagram p=2}.
	Then we consider the image of ${^U}d_7(x_1c_2^2)$ or ${^U}d_7(2x_1c_4)$, depending on $n\equiv 2$ or $0$ mod 4, in $^TE$ via the spectral sequence map $\Psi^*:{^U}E\to{^T}E$.
	To do this, let $v_i'=v_i-v_n\in H^*(BT^n)\cong\Zz[v_1,\dots,v_n]$ for $1\leq i\leq n$. By Proposition \ref{prop:T}, they are permanent cocycles in $^TE$.
	By definition, in $\Zz[v_1,\dots,v_n]$  we have
	\begin{equation}
		\begin{split}
\sigma_j(v_1,\dots,v_n)&=\sum_{1\leq i_1<\cdots< i_j\leq n}(v_{i_1}'+v_n)(v_{i_2}'+v_n)\cdots(v_{i_j}'+v_n)\\
&=\sum_{1\leq i_1<\cdots< i_j\leq n}\,\sum_{k=0}^j\sigma_k(v_{i_1}',\dots,v_{i_j}')v_n^{j-k}\\
&=\sum_{k=0}^j\binom{n-k}{j-k}\sigma_k(v_1',\dots,v_n')v_n^{j-k}.
		\end{split}	
\end{equation}
Hence $\Psi^*(x_1c_2^2)=x_1\sigma_2^2=x_1[\binom{n}{2}v_n^2+(n-1)\sigma_1'v_n+\sigma_2']^2$, where for simplicity we write $\sigma_j$ and $\sigma_j'$ instead of $\sigma_j(v_1,\dots,v_n)$ and $\sigma_k(v_1',\dots,v_n')$ respectively.
Similarly, \[\Psi^*(2x_1c_4)=2x_1\big[\binom{n}{4}v_n^4+\binom{n-1}{3}\sigma_1'v_n^3+\binom{n-2}{2}\sigma_2'v_n^2+(n-3)\sigma_3'v_n+\sigma_4'\big].\]
Since $x_1v_n^{k}={^T}d_3(\frac{1}{k+1}v_n^{k+1})$ for $k$ even  and $2x_1v_n={^T}d_3(v_n^2)$, $\Psi^*(x_1c_2^2)$ in $^TE_4$ is homologous to $2\binom{n}{2}(n-1)x_1\sigma_1'v_n^3$. For the same reason $\Psi^*(2x_1c_4)$ in $^TE_4$ is homologous to $2\binom{n-1}{3}x_1\sigma_1'v_n^3$.

If $n\equiv 2$ mod $4$, then $2\nmid\binom{n}{2},\,(n-1)$. Hence by Proposition \ref{prop:differentials K} and \ref{prop:T} we have
\[\Psi^*\circ {^U}d_7(x_1c_2^2)={^T}d_7\circ\Psi^*(x_1c_2^2)=y_{2,1}\sigma_1'=y_{2,1}\sigma_1\in {^T}E_7^{10,2}\subset{^T}E_3^{10,2}.\]
The equality $y_{2,1}\sigma_1'=y_{2,1}\sigma_1$ comes from the fact that $y_{2,1}$ is $2$-torsion and $\sigma_1=\sigma_1'+nx_n$, and the inclusion ${^T}E_7^{10,2}\subset{^T}E_3^{10,2}$ follows by the same argument as in the proof of part (b). This implies that ${^U}d_7(x_1c_2^2)=y_{2,1}c_1$ and (c) holds in this case. On the other hand, if $4\mid n$, then  $2\nmid\binom{n-1}{3}$. Hence the same reasoning shows that $\Psi^*\circ {^U}d_7(2x_1c_4)=y_{2,1}c_1$, completing the proof of (c).
\end{proof}

The desired group structure of $_2H^{11}(BPU_n)\cong {^UE}_\infty^{3,8}$ follows from Lemma \ref{lem:E^{3,8}} and Lemma \ref{lem:d_7}, since we already know that $_2H^{11}(BPU_n)$ is a cyclic group. 

Now we consider the ring structure of $H^{\leq 11}(BPU_n)$. First, we give a lemma for general $n\geq 4$, which is essentially the same as \cite[Lemma 4.3]{ZhaZ24}.
\begin{lem}\label{lem:kernel p=2}
	Let $n\geq 4$ be an integer. Then for the action of $\nabla$ on $\Lambda_n$, we have $\ker\nabla^4\cong \Zz^2$ with two generators $\alpha_2^2$ and $\alpha_4$, where 
	\begin{gather*}
		\alpha_2=\frac{1}{\epsilon_2(n-1)}\big[2n\sigma_2-(n-1)\sigma^2_1\big],\\
		\begin{split}
			\alpha_4=&\frac{1}{\epsilon_3(n)}\big[n\sigma_4-(n-3)\sigma_3\sigma_1-\frac{1}{2}(n^2+n+1)(n-2)(n-3)\sigma_2^2\\
			&+\frac{1}{2}n^2(n-2)(n-3)\sigma_2\sigma_1^2-\frac{1}{8}n(n-1)(n-2)(n-3)\sigma_1^4\big].
		\end{split}
	\end{gather*}
\end{lem}
\begin{proof}
It is easily verified that the coefficients of the monomial terms in the expression of $\alpha_4$ are all integers. Since $\nabla(\sigma_i)=(n-i+1)\sigma_{i-1}$, $\alpha_2,\alpha_4\in\ker\nabla$ follows by straightforward computation. The group isomorphism $\ker\nabla^4\cong \Zz^2$ is obvious, so to prove the lemma, it suffices to show that for any prime $p$,  $\rho(\alpha_2^2)$ and $\rho(\alpha_4)$ are linearly independent over $\Zz/p$, where $\rho:\Lambda_n\to\Lambda_n\otimes\Zz/p$ is the mod $p$ reduction.
This can be seen from the fact that the coefficients of $\sigma_2^2$ and $\sigma_1^4$ (resp. $\sigma_4$ and $\sigma_3\sigma_1$) in the expression of $\rho(\alpha_2^2)$  (resp. $\rho(\alpha_4)$) can not be both zero for all prime $p$. 
\end{proof}

\begin{prop}\label{prop:p=2 11}Let $\alpha_2$, $\alpha_4\in H^*(BU_n)\cong \Lambda_n$ be the elements defined in Lemma \ref{lem:kernel p=2}. Then
\[{^UE}_\infty^{3,8}\cong\begin{cases}
		\Zz/2\{x_1\alpha_2^2\}, &\text{if }n\equiv 2\text{ mod }4,\\
		\Zz/4\{x_1\alpha_4\}, &\text{if }n\equiv 4\text{ mod }8,\\
		\Zz/2\{x_1\alpha_4\}, &\text{if }n\equiv 0\text{ mod }8.
	\end{cases}\]

\end{prop}

\begin{proof}
Since $\alpha_2,\alpha_4\in{^U}E_4^{0,*}$, $\alpha_2$ and $\alpha_4$ are permanent cocycles by Theorem \ref{thm:E_4}, and so the elements $x_1\alpha_2^2$ and $x_1\alpha_4$ are also permanent cocycles in $^UE$ by Leibniz rule.  Hence it suffices to show that in ${^UE}_4^{3,8}$, $x_1\alpha_2^2\neq 0$ for $n\equiv 2$ mod $4$, that $2x_1\alpha_4\neq 0$ for $n\equiv 4$ mod $8$, and that $x_1\alpha_4\neq 0$ for $n\equiv 0$ mod $8$.

For the case $n\equiv 2$ mod $4$, let $\bar\Delta: BPU_2\to BPU_n$ be the diagonal map. Recall that $H^*(BPU_2)\cong H^*(BSO_3)\cong \Zz[p_1,W_3]/(2W_3)$, where $p_1$, $W_3$ are respectively the first Pontryagin class and the degree-$3$ integral Stiefel-Whitney class of $BSO_3$. By \cite[Lemma 7.2]{Gu21}  $\bar\Delta^*(\alpha_2)=(\frac{n}{2})^2p_1$ and $\bar\Delta^*(x_1)=W_3$, it follows that $\bar\Delta^*(x_1\alpha_2^2)=W_3p_1^2\neq 0$, and so $x_1\alpha_2^2\neq 0$.   

For the case $n\equiv 4$ mod $8$, let $\bar\Delta: BPU_4\to BPU_n$ be the diagonal map. Using \cite[Theorem 2.3]{Fan24} and making a diagonal argument as in the second paragraph of the proof of Lemma \ref{lem:p=3}, one can show that there exists an element $\alpha\in {^U}E_4^{0,8}$ such that $\bar\Delta^*(2x_1\alpha)\neq 0$, and then $2x_1\alpha\neq 0$. Since ${^U}E_4^{0,8}$ is generated by $\alpha_4$ and $\alpha_2^2$ by Lemma \ref{lem:kernel p=2} and $x_1\alpha_2=0$ by \cite[Theorem 1.1]{Gu21}, we obtain $2x_1\alpha_4\neq 0$.

For the case $n\equiv 0$ mod $8$, note that the coefficient of $\sigma_2^2$ in the expression of $\alpha_4$ is an odd number. However, if $k\sigma_2^2$ ($k\neq 0$) appears in the expression of $\nabla(\mm)$ for a monomial $\mm\in\Lambda_n=\Zz[\sigma_1,\dots,\sigma_n]$, then $\mm=\sigma_3\sigma_2$ or $\sigma_2^2\sigma_1$, and the coefficients of $\sigma_2^2$ in both $\nabla(\sigma_3\sigma_2)$ and $\nabla(\sigma_2^2\sigma_1)$ are  even numbers. This implies that $x_1\alpha_4\neq 0$ in $^UE_4^{3,8}$. The proof is completed.
\end{proof}

Recall that the ring structure of $H^{\leq 10}(BPU_n)$ is given in \cite{Gu21} where it is shown that $^UE_\infty^{0,\leq 10}$ is isomorphic to $\Zz[\alpha_2,\dots,\alpha_{j_n}]$, $j_n=\min\{5,n\}$, $\deg(\alpha_i)=2$. To determine the cup product structure in dimension $11$, we first choose elements $\w\alpha_i\in H^{2i}(BPU_n)$, $2\leq i\leq \min\{5,n\}$, such that their images in $^UE_\infty^{0,2i}$ are $\alpha_i$ and the product relations in \cite[Theorem 1.1]{Gu21} are satisfied. From Theorem \ref{thm:p-torsion} and Proposition \ref{prop:p=2 11} we know that 
\begin{equation}\label{eq:dim 11}
	H^{11}(BPU_n)\cong\begin{cases}
    \Zz/\epsilon_3(n)\{x_1y_{3,0}\}, &\text{if }n\not\equiv 0\text{ mod }2,\\
	\Zz/2\{x_1\w\alpha_2^2\}\oplus\Zz/\epsilon_3(n)\{x_1y_{3,0}\}, &\text{if }n\equiv 2\text{ mod }4,\\
	\Zz/4\{x_1\w\alpha_4\}\oplus\Zz/\epsilon_3(n)\{x_1y_{3,0}\}, &\text{if }n\equiv 4\text{ mod }8,\\
	\Zz/2\{x_1\w\alpha_4\}\oplus\Zz/\epsilon_3(n)\{x_1y_{3,0}\}, &\text{if }n\equiv 0\text{ mod }8.
\end{cases}
\end{equation}
Hence we only need to verify that $\mu(n)x_1\w\alpha_4=0$ in $H^*(BPU_n)$. From \eqref{eq:dim 11} we know that $\mu'(n)x_1\w\alpha_4\in \Zz/\epsilon_3(n)\{x_1y_{3,0}\}$, where $\mu'(n)=\mu(n)$ if $n\not\equiv 2$ mod $4$, and $\mu'(n)=2\mu(n)$ otherwise.
Thus we can replace $\w\alpha_4$ by $\w\alpha_4+ay_{3,0}$ for a suitable $a\in\Zz$ so that  $\mu'(n)x_1\w\alpha_4=0$. Finally, if $x_1\w\alpha_4\neq 0$ for some $n\equiv 2$ mod $4$, we can replace $\w\alpha_4$ by $\w\alpha_4+\w\alpha_2^2$ to get $x_1\w\alpha_4=0$. The proof of Theorem \ref{thm:below 12} is completed.

\bibliography{M-A}
\bibliographystyle{amsplain}
\end{document}